\documentclass[12pt]{amsart}
\usepackage[utf8]{inputenc}

\usepackage[T1]{fontenc}

\usepackage{enumitem}
\usepackage{amsmath,amsthm,amssymb}
\usepackage[english,french]{babel}
\usepackage{times}
\usepackage[uppercase=upright,greeklowercase=upright]{mathdesign}
\usepackage{mathrsfs}
\usepackage[
bookmarks=true,backref,
colorlinks=true,
citecolor=blue,
urlcolor=blue,linkcolor=magenta,
hypertexnames=false,
	hyperindex,
	pagebackref,
	breaklinks=true,
	bookmarks=false,
	colorlinks,
	linkcolor=red,
	citecolor=blue,
	urlcolor=orange,]{hyperref}

\usepackage{pdfcomment}			%
\usepackage[inline]{asymptote}	

\usepackage{enumitem}

\usepackage{tikz-cd}
\usepackage{amsrefs}

\usepackage{latexsym,graphicx}

\usepackage{setspace}

\usepackage{amsmath,amssymb,amscd,stmaryrd,mathtools,bm,rotating}
\usepackage[all]{xy}

\author
{Rodolphe~\textsc{Richard}, Emmanuel Ullmo}
\address{Rodolphe \textsc{Richard}\\
13, rue du Croisic\\
22200 Plouisy,
Bretagne, \textsc{France}
}
\email{rodolphe.richard@normalesup.org}

\address{Emmanuel  \textsc{Ullmo}\\
IHES, Université Paris Saclay, Laboratoire Alexandre Grothendieck CNRS
}
\email{ullmo@ihes.fr}
\address{Jiaming  \textsc{Chen}\\
Institut Elie Cartan, Université de Lorraine, Site de Nancy, B.P. 70239, F-54506 Vandoeuvre-les-Nancy Cedex France }
\email{jiaming.chen@univ-lorraine.fr}

\newtheorem{thm}{Theorem}[section]

\numberwithin{equation}{section}
\newtheorem{theorem}[thm]{Théorème}
\newtheorem{corollaire}[thm]{Corollaire}
\newtheorem{definition}[thm]{Définition}

\newtheorem{proposition}[thm]{Proposition}
\newtheorem{lemme}[thm]{Lemme}
\theoremstyle{definition}
\newtheorem{remarque}[thm]{Remarque}
\newtheoremstyle{explication}
  {\topsep}
  {\topsep}
  {\itshape}
  {1em}
  {\bfseries}
  {}
  {0pt}
  {}
\theoremstyle{explication}

\newcommand{\Q}{\mathbb{Q}}
\newcommand{\R}{\mathbb{R}}

\newcommand{\Z}{\mathbb{Z}}
\newcommand{\N}{\mathbb{N}}

\def\SS{{\mathbb S}}
 
\def\VV{{\mathbb V}}  
\newcommand{\QQ}{\mathbb{Q}}
\newcommand{\RR}{\mathbb{R}}
\newcommand{\CC}{\mathbb{C}}
\newcommand{\ZZ}{\mathbb{Z}}
\newcommand{\NN}{\mathbb{N}}

\newcommand{\cHH}{{\mathcal H}}

\newcommand{\cE}{{\mathcal E}}
\newcommand{\cF}{\mathcal{F}}

\DeclareMathOperator{\sous}{\backslash}

\newcommand{\G}{\mathbf{G}}

\newcommand{\HH}{\mathbf{H}}

\newcommand{\wV}{\widehat{V}}

\def\Dcal{{\mathcal D}}
\def\Ecal{{\mathcal E}}
\def\Fcal{{\mathcal F}}
\def\Hcal{{\mathcal H}}

\def\Gbold{{\pmb G}}
\def\Hbold{{\pmb H}}

\def\Gam{{\Gamma}}

\def\bs{{\backslash}}

\def\dprime{{\prime\prime}}

\newcommand\supp{\operatorname{supp}}
\newcommand\ad{\operatorname{ad}}

\newcommand\GL{\pmb{\operatorname{GL}}}
\newcommand\Zcent{\operatorname{\bf Z}}

\newcommand{\der}{\textnormal{der}}
\newcommand{\WS}{\textnormal{WS}}
\newcommand{\Zar}{\textnormal{Zar}}
\newcommand{\alg}{\textnormal{alg}}
\newcommand{\an}{\textnormal{an}}

\newcommand{\SL}{\mathbf{SL}}

\newcommand{\PP}{\mathbf{P}}
\newcommand{\NNN}{\mathbf{N}}

\newcommand{\vertrestriction}{\mathrel{|}}
\renewcommand{\restriction}{\vertrestriction}

\title[Équidistribution et O-minimalité]{Équidistribution de sous-variétés faiblement spéciales\\ et O-minimalité: André-Oort géométrique}
\date{\today}
\begin{document}

\setcounter{tocdepth}{1}

\begin{abstract} Une caractérisation des sous-variétés des variétés de Shimura qui contiennent un sous ensemble Zariski dense de sous-variétés faiblement spéciales est démontrée dans \cite{Ullmo3} en combinant des résultats d'o-minimalité et des résultats de transcendance fonctionnelle (Théorème d'Ax-Lindemann hyperbolique~\cite{KUY}, obtenu peu après).
Nous obtenons dans ce texte une nouvelle preuve de cet énoncé par des techniques de dynamique sur les espaces homogènes dans l'esprit de  travaux antérieurs de Clozel et du second auteur \cite{ClozelUllmo1}, \cite{Ullmo2}. La preuve combine de la théorie ergodique à la Ratner,
à la Mozes-Shah complété par Daw-Gorodnick-Ullmo, avec un énoncé de~\cite{ActesDries} sur la dimension d'une limite de Hausdorff
d'une suite de sous-ensembles définissables (dans une théorie o-minimale) extraite d'une famille définissable. On obtient au passage des énoncés de dynamique homogène généraux valables sur des quotients arithmétiques arbitraires qui sont d'un intérêt indépendant, s'applicant par exemple dans l'étude des variations de
structures de Hodge.
\end{abstract}

\selectlanguage{french}
\maketitle
\centerline{Appendice avec Jiaming Chen: {\bf Equidistribution des sous-vari\'et\'es faiblement sp\'eciales }}
\centerline{\bf{horizontales dans les domaines de p\'eriodes }}
\tableofcontents

\section{Introduction}
\subsection{Partie géométrique de la conjecture d'André-Oort}

\subsubsection{Sous-vari\'et\'es faiblement sp\'eciales.}\label{s2}Introduisons la notion de sous-variétés faiblement spéciale d'une variété arithmétique.\footnote{Les résultats en vue seront insensibles au passage à un réseau commensurable, ni au fait que~$\Gamma$ est de congruence ou non.
On pourra choisir~$\G$ semisimple adjoint si on veut,  et~$S$ n'a pas besoin d'être une variété de Shimura \emph{stricto sensu}.}

Soit $X$ un espace hermitien symétrique.  Soient $\G$ un groupe algébrique réductif  sur~$\QQ$ et $K_{\infty}$  un sous-groupe compact maximal de $\G(\RR)$,   tels que $X=\G(\RR)/K_{\infty}$. Soit $\Gamma\subset \G(\QQ)$ un réseau arithmétique sans torsion de $\G$ et $S=\Gamma\backslash X$. Alors $S$ est un espace localement symétrique hermitien qui est muni canoniquement d'une structure de variété algébrique quasi-projective par les résultats de Baily-Borel, dite \emph{variété arithmétique}.

\begin{definition}
Les sous-vari\'et\'es \emph{faiblement sp\'eciales} de $S$ sont les sous-vari\'et\'es
alg\'ebriques irr\'eductibles de $S$ qui sont  totalement g\'eod\'esiques dans $S$.
\end{definition}
 D'apr\`es les travaux de Moonen \cite{Moonen1} toute sous-variété faiblement spéciale de $S$ peut s'écrire sous la forme
\begin{equation}\label{Eq-fs}
Z=[H,x]:=\Gamma\backslash\Gamma\cdot \HH(\RR)^+\cdot x\subseteq S
\end{equation}
pour un $\QQ$-sous-groupe algébrique réductif~$\HH$ de $\G$ et un point~$x\in X$, tels que~$X_H:=\HH(\RR)^+\cdot x$ est un sous-espace hermitien symétrique de $X$. 
Alors~$\Gamma_H=\Gamma \cap \HH(\QQ)^+$ est un réseau arithmétique de $\HH$ et au morphisme fini
\(\Gamma_H\sous X_H\to Z\) près,~$Z$ est une variété arithmétique.

%
%

\subsubsection{Sous-variétés produits faiblement spéciaux}\label{intro produits}

Dans~\eqref{Eq-fs}, si~$\HH^{ad}=\HH_1^{ad}\times \HH_2^{ad}$ est une factorisation en produit de deux groupes semi-simples pour des sous-groupes réductifs $\HH_1$ et $\HH_2$ de $\HH$, alors~$X_{H_1}:=\HH_1(\RR)^+\cdot x$ et~$X_{H_2}:=\HH_2(\RR)^+\cdot x$ sont hermitiens symétriques. De plus	
$[H_1,x]:=\Gamma\backslash\Gamma\cdot \HH_1(\RR)^+\cdot x$ et~$[H_2,x]:=\Gamma\backslash\Gamma\cdot \HH_2(\RR)^+\cdot x$ 
sont des sous-variétés faiblement spéciales de~$S$. Il y a une factorisation
~$X_H\simeq X_{H_1}\times X_{H_2}$ et
$\Gamma_{H_1}\times \Gamma_{H_2}$ est d'indice fini dans $\Gamma_H$  et le morphisme 
$$
\psi: S_1\times S_2:=\Gamma_{H_1}\backslash X_{H_1}\times  \Gamma_{H_2}\backslash X_{H_2}\longrightarrow 
\Gamma_H\sous X_H \to Z\hookrightarrow S.
$$
est le composé de  deux morphismes finis suivis d'une injection.

\begin{definition}On dit qu'une sous-variété algébrique irréductible de~$S$ est \emph{un  produit faiblement spécial} si
elle se met sous la forme, 
\[
V=\psi(S_1\times V_2),
\]
pour une sous-variété algébrique irréductible~$V_2$ de~$S_2$
avec~$\dim(S_1)$, $\dim(S_2)$ et~$\dim(V_2)$ non nulles.
\end{definition}

\subsubsection{André-Oort géométrique}

Si $V$ est faiblement spéciale de dimension non nulle, ou si $V$ est un produit faiblement spécial, l'ensemble de sous-variétés faiblement spéciales de dimension positive de~$S$ contenues dand $V$ est Zariski dense dans $V$. 

Un résultat du second auteur de ce texte \cite{Ullmo3} qui utilise  le Théorème d'Ax-Lindemann hyperbolique prouvé dans~\cite{KUY} montre que cela caractérise les sous-variétés irréductibles de $S$ qui contiennent un sous-ensemble Zariski-dense de sous-variétés faiblement spéciales de dimension positive. C'est la partie géométrique de la conjecture d'André-Oort:
 
 \begin{theorem}\label{T1}
Une sous-variété algébrique irréductible $V$ de $S$ qui contient un sous-ensemble Zariski-dense de sous-variétés faiblement spéciales, toutes de dimension non nulle, est une sous variété faiblement spéciale ou un  produit faiblement spécial.
 \end{theorem}
 
\subsubsection{Approche}\label{sApproche}
Nous donnons dans ce texte, une méthode alternative pour obtenir ce résultat. Notre approche  combine des idées de dynamique homogène avec des arguments  d'o-minimalité qui sont nouveaux dans ce contexte. Nous obtenons dans l'appendice de ce papier rédigé en commun avec Jiaming Chen une extension de ce résultat dans le contexte général des $\ZZ$-variations de structures de Hodge.

Cette méthode s'inscrit dans la continuation de travaux initiés par Clozel et le second auteur de ce texte \cite{ClozelUllmo1}, \cite{ClozelUllmo2}.
 La technologie sur lesquels s’appuie cette approche nous vient de la première moitié des années 1990: les Théorèmes de Ratner sur la conjecture de Raghunathan  affirmant la rigidité des flots unipotents \cite{Ra91a},  un article de Mozes-Shah \cite{MoSh} qui précise la limite d'une suite convergente donnée de mesures homogènes, en application des Théorèmes de Ratner et les résultats de Dani-Margulis \cite{DaMa} qui permettent de contrôler les phénomènes de perte de masse dans ces questions. En dynamique homogène, nous requérons en outre un résultat récent de~\cite{DGU2} prolongeant~\cite{MoSh}, détaillé en~\ref{sDGU}.




\subsection{Dynamique homogène et Stratégie}\label{strategie}
 Notre approche dynamique nous ammène au Théorème~\ref{T1} par un Théorème~\ref{T2} que l'on peut considérer comme son alter ego dynamique. 
 
Les sous-variétés faiblement spéciales de $S$ sont munies canoniquement d'une mesure de probabilité grâce à leur interprétation en terme d'espaces localement symétriques hermitiens. 
On considère plus généralement dans ce texte des \emph{sous-ensembles homogènes} de $S$. Ce sont les ensembles analytiques fermés de $S$ de la forme 
\begin{equation}\label{Eq-HS}
Z=\Gamma\backslash\Gamma \HH(\RR)^+\cdot x \mapsto S
\end{equation}
pour un sous-groupe algébrique $\HH$ de $\G$ sans caractère rationnel de sorte que $\Gamma_H$ est un réseau de~$\HH$.  

On munit~$Z$ d'une mesure de probabilité~$\mu_Z$, l'image directe de la mesure de probabilité invariante sur~$\Gamma_H\sous H$
par
\[(\Gamma\cap H)\sous H \to\Gamma\sous\Gamma H\to \Gamma \sous \Gamma\cdot H\cdot x\subseteq\Gamma \sous  X\]
Une telle mesure, dans~$S$, sera dite \emph{homogène} dans la suite. Si le choix du couple~$[H,x]$ dans l'écriture~$Z=[H,x]$  n'est pas unique, la mesure~$\mu_Z$ ne dépend que de~$Z$.

Ce texte suit la démarche de Clozel et du second auteur qui consiste à étudier des suites arbitraires de sous-variétés faiblement spéciales de dimension non nulle via les suites de mesures de probabilités associées. 

Il sera utile de disposer des définitions suivantes concernant des suites de sous-variétés algébriques de $S$ et des suites de sous-variétés faiblement spéciales.

\begin{definition} Soit $V$ une sous-variété algébrique irréductible de $S$.
Une suite~$(Z_n)_{n\geq 0}$ de sous-variétés irréductibles~$Z_n$ de $V$ est
dite
\begin{itemize}
\item \emph{Zariski dense} si aucune sous variété \emph{fermée et propre} de $V$ ne la contient;
\item \emph{Hodge générique} si aucune sous variété \emph{spéciale}  propre de $S$ ne la contient;
\item \emph{générique} si toute sous-suite infinie est Zariski dense (le point générique de~$Z_n$ converge vers le point générique dans le schéma~$V$);
\item \emph{stricte} si toute sous-suite infinie est Hodge générique.
\end{itemize}
\end{definition}

\begin{definition}
Une suite de sous-variété faiblement spéciales~$Z_n$ est dite
\begin{itemize}
\item \emph{convergente en mesure} (dans un espace~$\overline{S}$) si la suite des mesures canoniques~$\mu_n=\mu_{Z_n}$ admet une limite dans l'espace des mesures de probabilités  dans l'espace topologique~$\overline{S}$ contenant~$S$: par exemple la compactification de Borel-Serre, de Satake maximale, ou de Bailly-Borel de~$S$.  
\item \emph{tendue (en mesure) } si les conditions du Théorème de Prokhorov s'appliquent à cette suite des mesures canoniques. Un critère équivalent est donné par la théorie de Dani-Margulis:
il existe un compact de~$S$ qui rencontre chacune des~$Z_n$.
\end{itemize}
\end{definition}

La difficulté centrale que la méthode de ce texte permet de surmonter se rencontre déjà dans le cas où la suite $(Z_n)_{n\in\NN}$ est tendue.  Pour traiter le cas non tendu nous utilisons les résultats récents  de C.~Daw, A.~Gorodnick et le second auteur de ce texte \cite{DGU1}, \cite{DGU2}, qui décrivent le comportement des suites de mesures homogènes sur les compactifications maximales de Satake  des espaces localement symétriques. 

Dans l'approche initiée par Clozel et le second auteur, on étudiait une suite de sous-variétés spéciales $(Z_n)_{n\in \NN}$ de $S$. En faisant des hypothèses supplémentaires, $Z_n$ \emph{fortement spéciale} dans  \cite{ClozelUllmo1}, et plus généralement~\emph{non facteur} dans  \cite{ClozelUllmo2}, on concluait en premier que $(Z_n)_{n\in \NN}$ était tendue, et ensuite que toute limite faible $\mu_{\infty}$ de la suite de mesures
 $(\mu_{Z_n})_{n\in \NN}$ était la mesure associée à une sous-variété spéciale $Z_{\infty}$ (et même respectivement fortement spéciale ou non facteur). La conclusion topologique que l'on obtenait,  dans l'esprit du Théorème \ref{T1}, reposait sur le fait que si la suite des $\mu_{Z_n}$ convergeait vers $\mu_{Z_{\infty}}$, alors on pouvait montrer l'inclusion $Z_n\subset Z_{\infty}$ pour tout $n$ assez grand.

Quand on travaille, comme dans ce texte, avec une suite arbitraire $(Z_n)_{n\in\NN}$,  de sous-variétés faiblement spéciales de $S$, on est confronté à trois difficultés principales.
\begin{itemize}
\item Si $\mu_{\infty}$ est une limite faible de la suite $(\mu_{Z_n})_{n\in \NN}$, alors en général $\mu_{\infty}$ n'est pas une mesure associée à une sous-variété faiblement spéciale $Z_{\infty}$.
\item Si $(\mu_{Z_n})_{n\in \NN}$ converge faiblement vers $\mu_{\infty}$, alors en général le support de $\mu_{\infty}$ ne contient pas les $Z_n$ pour $n$ assez grand. 
\item Dans le cas non tendu, la suite~$(\mu_{Z_n})_{n\in \NN}$  peut converger faiblement vers la mesure nulle.
\end{itemize}

Pour simplifier l'exposition, nous décrivons la méthode dans le cas tendu.

Nous contournons la première difficulté à l'aide de raisonnements classiques en dynamique homogène combinés avec un \emph{"baby case"} du Théorème d'Ax-Lindemann hyperbolique (Proposition \ref{BAL}). Nous partons de la description~\eqref{Eq-fs}, 
$$
Z_n=\Gamma\backslash\Gamma \cdot \HH_n(\RR)^+\cdot x_n
$$
pour un $\Q$-sous-groupe r\'eductif $\HH_n$ de $\G$ et un point $x_n$, que l'on peut choisir dans un ensemble fondamental~$\cF\subset X$ convenable
fixé. Quand la suite $(Z_n)_{n\in\NN}$ est tendue, la théorie de Ratner permet de montrer qu'en passant éventuellement à une sous-suite, $\mu_n$ converge étroitement vers une mesure \emph{homogène}~$\mu_{\infty}$ et
la suite $(x_n)_{n\in \N}$ converge vers $x_{\infty}\in \cF$. De plus
$$
Z'_{\infty}:={\rm Supp}(\mu_{\infty})= \Gamma_{H_\infty}\backslash \HH_\infty(\RR)^+\cdot x_\infty
$$
pour un sous-groupe algébrique $\HH_{\infty}$ de $\G$ tel que $\HH_n\subset \HH_{\infty}$ pour tout $n$ assez grand. Le \emph{ "baby case"} du Théorème d'Ax-Lindemann hyperbolique nous assure que la clôture de Zariski $Z_{\infty}$ de $Z'_{\infty}$ est une sous-variété faiblement spéciale de $S$. En particulier si une sous-variété algébrique $V$ de $S$ contient les $Z_n$, alors elle contient~$Z_{\infty}$.

L'apport principal de se travail est le résultat suivant qui permet de surmonter la deuxième difficulté.   
L'essence de cet énoncé réside en la confrontation de l'inégalité de dimensions
\[
\dim(Z'_\infty)\geq \limsup (\dim(Z_n))_{\N},
\]
évidente dans l'univers de la dynamique homogène à la Ratner, mise en regard de l'inégalité en sens inverse
\[
\dim(Z'_\infty\cap V)\leq \limsup (\dim(Z'_n)\cap V)_{\N},
\]
naturelle dans le monde des famille définissables dans une théorie o-minimale, où la géométrie varie de façon modérée.


\begin{theorem}\label{T2}
Soit $V$ une sous-variété algébrique irréductible de $S$ qui contient une suite  $(Z_n)_{n\in\N}$ de sous-variétés faiblement spéciales. 

En passant éventuellement à une sous-suite extraite, il existe des écritures
\[
Z_n=\Gamma\backslash\Gamma \HH_n(\RR)^+\cdot x_n
\] pour des sous-groupes semi-simples $\HH_n$ de $\G$ de telle sorte qu'il existe un sous-groupe algébrique~$\HH_{\infty}$ sans caractère rationnel  tel que pour $n$, $\HH_n\subset \HH_{\infty}$ et tel que~$V$ contient les sous-espaces homogènes 
\[
Z'_n:=\Gamma\backslash\Gamma\cdot \HH_\infty(\RR)^+\cdot x_n.
\]
\end{theorem}
L'observation capitale est que le sous-ensemble homogène~$Z'_n$ contient la sous-variété faiblement spéciale~$Z_n$.
 Le groupe~$\HH_\infty$ étant généralement strictement plus grand que les~$\HH_n$, l'inclusion de $Z_n$ dans $Z'_n$ peut être stricte.

On peut encore utiliser le ``baby case'' du Théorème d'Ax-Lindemann hyperbolique.
\begin{corollaire}
La sous-variété algébrique~$V$ contient les sous-variétés faiblement spéciales~$\tilde{Z}_n=\overline{Z'_n}^{Zar}$.
\end{corollaire}

Le Théorème~\ref{T1} se démontre en
 partant d'une suite de $Z_n$ qui sont maximales parmi les sous-variétés faiblement spéciales contenues dans~$V$.
 On obtient alors que $\tilde{Z}_n=Z_n$ ce qui force
$\HH_n=\HH_\infty$ pour tout $n$. La conclusion recherchée s'obtient avec la description des produits faiblement spéciaux 
décrit dans \ref{intro produits} en posant~$\HH_1=\HH_\infty$ et~$\HH_2=Z_\G(\HH_1)$.

\subsection{Quotients arithmétiques  généraux}\label{QAG}

Le c\oe{}ur technique de ce travail est suffisamment souple pour s'appliquer à des espaces localement homogènes plus généraux
que les variétés arithmétiques et en particulier aux domaines de périodes associés aux  $\ZZ$-variations de structures de Hodge.

On fixe encore un groupe semi-simple $\G$, un réseau arithmétique sans torsion~$\Gamma\subset \G(\Q)$ et un sous-groupe compact $M$ de $G=\G(\R)^+$.
Notons que nous ne faisons pas d'hypothèses sur $M$, on peut prendre $M= \{1\}$ ou $M=K_{\infty}$. Il n'y a plus non plus d'hypothèses hermitiennes sur l'espace symétrique de $G$, on peut par exemple penser à $\G=\SL_n$.
Alors le \emph{quotient arithmétique}
$$
S=S_{\Gamma,G,M}:=\Gamma\backslash G/M
$$
a une structure de variété analytique réelle.

Il est possible, pour chaque choix de compact maximal~$K_\infty$ contenant~$M$, de munir naturellement~$S_{\Gamma,G,M}$ d'une structure d'espace semi-algébrique\footnote{
Elle est uniquement déterminée par: si~$\cF$ est un ensemble (resp. domaine) fondamental de l'action de~$\Gamma$ sur~$G$ union finie de domaines de Siegel rationnels (et relatifs à~$K_\infty$) semi-algébriques dans~$G$, alors la surjection (resp. bijection)~$\cF/M \to S$ est semi-algébrique. Cela dépend fortement du choix de~$K_\infty$.
} réel sur $S$, comme observé par~\cite[Th.~1.1]{BKT}. Pour toute structure~$o$-minimale, telle  $\RR^{an}$ et $\RR^{an,exp}$,
nous aurons une notion de sous-variétés différentielle définissable dans cette structure.

Pour travailler dans le cadre de la dynamique homogène à la Ratner, on définit classiquement la classe $\cHH$ de $\Q$-sous-groupes algébriques de $\G$. 
Un sous-groupe algébrique~$\HH\leq \G$ défini sur $\Q$ est dit de type $ \mathcal{H}$ si son radical $\mathbf{R}_H$ est unipotent et si $\HH/\mathbf{R}_H$ est un produit presque direct de sous-groupes $\Q$-simples dont les points réels ne sont pas compacts. 
 
Une \emph{sous-ensemble homogène de $S$} désigne un sous-ensemble de la forme
$$
Z=\Gamma\backslash \Gamma \cdot H\cdot g\cdot M\subseteq S
$$
pour un groupe $\HH$ de type $\cHH$ et $g\in G$, avec $\Gamma_H=\Gamma\cap H$. 
Alors ces sous-ensembles homogènes sont à la fois sous-ensembles analytiques et, d'après le Théorème 1.1 de \cite{BKT}, sont semi-algébriques.

Nous obtenons dans ce cadre le résultat suivant:

\begin{theorem}\label{teo1.6}
 Soit~$V$ une sous-variété analytique définissable de~$S$ dans une structure~$o$-minimale.


Soit $(Z_n)_{n\in\N}$ une suite de sous-ensembles homogènes de $S$ contenus dans~$V$. En passant éventuellement à une sous-suite infinie extraite, il existe des écritures
\begin{equation}\label{theo1.6eq1}
Z_n=\Gamma\backslash\Gamma \HH_n(\RR)^+\cdot g_n\cdot M
\end{equation}
pour des sous-groupes $\HH_n$ de type $\cHH$ de  $\G$ et un sous-groupe algébrique~$\HH_{\infty}$ de type $\cHH$ telles que pour $n$, $\HH_n\subset \HH_{\infty}$ et telles  que  $V$ contient les sous-espaces homogènes
\begin{equation}\label{theo1.6eq2}
Z'_n:=\Gamma\backslash\Gamma \HH_\infty(\RR)^+\cdot g_n\cdot M.
\end{equation}
\end{theorem}



On verra que cet énoncé se ramène aisément du cas $M=\{1\}$.

Ceci dit une sous-classe de ces variétés $S_{G,K,M}$ jouent un rôle important comme domaines de périodes associés aux variétés algébriques complexes qui portent une $\ZZ$-variation de structures de Hodge. Notons que dans ces cas la structure d'espace $\R^{alg}$-définissable sur $S$ ainsi que ses extensions en des structures $\R^{an,exp}$ sont alors uniquement définies car il n'y a  qu'un unique choix de compact $K_{\infty}$ de $G$ contenant $M$. Les conséquences de ce texte dans ce cadre permettent de  complêter les travaux de Jiaming Chen ~\cite{C21}. L'analogue du Théorème $\ref{T1}$ est obtenu dans l'appendice~\ref{apendix} de ce texte écrit en collaboration avec Jiaming Chen.

\subsection{Organisation du texte}

La section \ref{sec2} donne des rappels de base de la théorie o-minimale. On souligne la notion de \emph{familles définissables} et on introduit les structures de variétés définissables dans une théorie o-minimale. 

La section~\ref{sec3} consiste en l'énonciation d'une propriété, trouvée dans~\cite{Actes Dries}, du comportement de la dimension d'une suite de compacts définissables dans un compact fixe de $\R^n$, au passage à la limite de Hausdorff, \emph{lorsque ces compacts sont définissables \og{}en famille\fg{}}. Alors \og{}la dimension ne peut que chuter à la limite\fg{}.

La section \ref{sec4} introduit les notions de théorie de la mesure qui nous serons utiles dans la suite. On y montre le résultat clef Théorème \ref{Proposition mesures}, technique mais taillé à nos besoin ultérieurs. C'est une élaboration directe depuis la section précédente, incorporant 
des suites étroitement convergentes de mesures dont le support est comparé à des suites de compacts définissables en famille,
et dont les valeurs sur des parties définissables ont de bonnes propriétés dimensionnelles.

La section \ref{sec5} récapitule les outils de dynamique homogène sur les quotients arithmétiques généraux: théorie de Ratner, travaux de Mozes-Shah, de Daw-Gorodnik-Ullmo, ainsi que sur les espaces localement symétriques hermitiens: mesures canoniques des sous-variétés faiblement  spéciales et des sous-ensembles homogènes, \og{}baby case\fg{} du Théorème d'Ax-Lindemann hyperbolique, définissabilité des applications d'uniformisation.
En outre nous prouvons en~\ref{sec6}, faute de référence, un énoncé de finitude en théorie de la réduction, pour en déduire que certaines familles de sous-ensembles de $G$ sont des \emph{familles} définissables. 

C'est à partir de la section \ref{sec8} que débute notre mariage des idées issues de la dynamique homogène et des propriétés des familles définissables dans une théorie o-minimale. En instanciant l'énoncé technique de~\ref{sec4} nous obtenons, d'abord dans le cadre des quotients arithmétiques généraux, le Théorème \ref{teo1.6}. Le Théorème~\ref{T2} s'obtient par ricochet, en nous restreignant au cadre des espaces localement symétrique hermitien, et en utilisant la définissabilité de l'uniformisation. Nous terminons par la section \ref{sec9} avec la preuve du Théorème \ref{T1}, en rédigeant l'argument  esquissée à la fin de~\ref{strategie}.

Enfin, l'appendice écrit en collaboration avec Jiaming Chen étend le Théorème \ref{T1} dans les domaines de périodes associés aux $\ZZ$-variations de structures de Hodge polarisés sur des variétés complexes quasi-projectives.

\section{Préliminaires}\label{sec2}
\subsection{o-minimalité}

La nouveauté de cet article est l'introduction d'un argument de nature o-minimale dans les questions d'équidistribution de mesures homogènes. Nous rappelons ici quelques définitions de base de la théorie utiles à la compréhension du texte. Plusieurs références classiques sur le sujet
peuvent aider le lecteur notamment \cite{Dries}, \cite{ActesDries}.

Une \emph{structure o-minimale} est la donnée pour chaque entier $n$,  d'une classe $\mathcal{S}_n$ de sous-ensembles de $\R^n$, dit \emph{définissables}, qui partagent beaucoup de propriétés avec les ensembles semi-algébriques, à savoir ceux définis par inéquations polynomiales: tel le disque unité obtenu par~$x^2+z^2\leq1$. Queques propriétés fondamentales sont:
\begin{itemize}
\item les sous-ensembles semi-algébriques sont définissables dans toute structure o-minimale;
\item la plupart des opérations courantes (union finie, complémentaire, produit cartésien fini, ...) conservent la définissabilité;
\item l'image d'un ensemble définissable par projection sur certaines coordonnées est définissable;
\item les composantes connexes d'un ensemble définissable sont en nombre fini.
\end{itemize}
On définit alors une \emph{application définissable} comme une application~$f:A\to B$ entre deux ensembles définissables dont le graphe, dans~$A\times B$ est définissable.

On combine flexibilité: toute partie décrite par une formule logique du premier ordre détermine un objet définissable; et rigidité: la condition de finitude limite la complexité des ensembles considérés. 
La notion d'ensemble définissable dans une structure o-minimale est suffisamment riche pour que beaucoup de notions de géométrie algébrique réelle s'appliquent:
Trois exemples de structures o-minimales qui nous importent sont
\begin{itemize}
\item celle, notée~$\R^{\text{alg}}$, des sous-ensembles semi-algébriques;
\item celle, notée~$\R^{\text{an}}$, engendrée par les graphes de restrictions à~$[0;1]^n\subseteq\R^n$ de fonctions sur $\R^n$ qui sont analytiques réelles sur un voisinage de $[0,1]^n$.
\item celle, notée~$\R^{\text{an,$\exp$}}$, engendrée par la précédente et le graphe de la fonction exponentielle. 
\end{itemize}

\begin{definition}
Soit $B\subset \R^r$ un ensemble définissable. Une famille $(A_b)_{b\in B}$ d'ensembles définissables $A_b\subset \R^s$ indexée par $B$ est une {\bf famille définissable} si l'ensemble
$$
\tilde{A}:=\cup_{b\in B} (\{b\}\times A_b)\subset \R^r\times \R^s
$$
est définissable. 

Une suite $(B_n)_{n\in \N}$ de sous-ensembles définissables de $\R^s$ est dite {\bf extraite de la famille définissable} $(A_b)_{b\in B}$ si pour tout $n\in \N$ il existe $b_n\in B$ tel que $B_n=A_{b_n}$.
\end{definition}

\subsection{Dimension}
La notion de \emph{dimension} 
$$
\dim(A)\in\{-\infty;0;\ldots;n\}
$$
 d'un ensemble définissable
$A\subseteq \R^n$ (\cite[]{Coste2}, \cite[]{Dries}, \cite[]{Actes Dries}) sera cruciale dans la suite.  La dimension des ensembles définissables dans une structure o-minimale vérifie la plupart des propriétés attendues . 
On suit la convention~$\dim(\emptyset)=-\infty$ (indépendamment de~$n$). Pour~$A\subseteq B\subseteq \R^n$ définissables, nous avons notamment, voir~\cite[Prop. 3.17, Th. 3.22]{Coste2},
\begin{equation}\label{dimension facts}
\dim(A)\leq \dim(B)\qquad \dim(\partial(A))\leq \dim(A)-1\text{ où }\partial(A):=\overline{A}\smallsetminus A.
\end{equation}

Faute de référence, nous terminons cette section par  deux énoncés, qui découlent de~\eqref{dimension facts}.  Le premier permettra de traduire en propriétés topologiques des arguments sur les dimensions. 
\begin{lemme}\label{Lemme dimension interieur}
Soit~$A\subseteq B$ deux ensembles définissables non vides de même dimension.
Alors~$A$ contient un ouvert non vide de~$B$.
\end{lemme}
\begin{proof}Prouvons la contraposée. Soit ainsi~$A\subseteq B$ d'intérieur vide dans~$B$. Autrement vu,~$U=B\smallsetminus A$ est dense dans~$B$.
Donc~$A\subseteq \overline{U}\smallsetminus U=\partial U$ (adhérence et bord pris dans~$B$ ou dans~$\R^n$ tel que~$B\subseteq\R^n$ peu importe). Enfin, utilisant~\eqref{dimension facts},
\[
\dim A\leq \dim \partial U\leq \dim(U)-1\leq \dim(B)-1.
\]
Comme~$B$ est non vide,~$\dim(B)$ est finie et~$\dim(B)-1<\dim(B)$. Finalement~$\dim(A)\neq \dim(B)$.
\end{proof}
Du résultat suivant, seule la toute dernière conclusion nous importe.
\begin{lemme}\label{definissable est borelien}
Soit~$n\in\N$ et soit~$A\subseteq\R^n$ une partie définissable.
Alors 
\[
A=\overline{A}\smallsetminus\left(\overline{\partial(A)}\smallsetminus\left(\overline{\partial(\partial (A))}\smallsetminus\left(\ldots\left(\overline{\partial^{\dim(A)}A}\smallsetminus\emptyset\right)\ldots\right)\right)\right),
\]
soit, en termes de fonctions caractéristiques,
\[
\chi_A=\chi_{\overline{A}}-\chi_{\overline{\partial A}}+\chi_{\overline{\partial\partial A}}-\ldots =\sum^{\dim(A)}_{c=0} (-1)^c\cdot \chi_{\overline{\partial^{c}A}}.
\]
En particulier~$A$ est dans la sous-algèbre d'ensembles engendrée par les parties définissables fermées~$B\subseteq \R^n$ de dimension~$\dim(B)\leq\dim(A)$.  

A fortiori~$A$ est dans la tribu de Borel de~$\R^n$.
\end{lemme}

Nous ne ferons qu'esquisser une démonstration.

\noindent On a noté par récurrence~$\partial^cA=\partial(\partial^{c-1}A)$ et~$\partial^0A=A$. Par passage à la clôture, on a la chaine d'inclusions.
$$
\overline{A}\supseteq\overline{\partial(A)}\supseteq \overline{\partial^2 A}\supseteq\ldots.
$$

Le raisonnement s'opère  par récurrence sur~$\dim(A)\in\{-\infty\}\cup\N$. On  a $A=\overline{A}\smallsetminus B$ avec~$B=\partial  A$ satisfaisant~$\dim(B)\leq \dim(A)-1$. On répète l'opération en choisissant~$A=B$. Pour~$B=\partial^{\dim(A)+1}A$, on trouvera~$\dim(B)\leq \dim(A)-(\dim(A)+1)<0$, donc~$\dim(B)=-\infty$ et $B=\emptyset$.

\subsection{Structure de variété définissable dans une théorie o-minimale}\label{Vom}

Soit $V$ une variété différentiable de dimension $n$. Soit $\mathcal{S}=(\mathcal{S}_n)_{n\in \N}$ une structure o-minimale. On dit que $V$ est définissable dans $\mathcal{S}$, si il existe un atlas fini de cartes  de $V$ tel que les applications de transitions soient définies dans $\mathcal{S}$. On dispose alors d'un recouvrement ouvert $(U_i)_{i\in I}$ de $V$, pour un ensemble fini d'indices $I$,  et pour tout $i\in I$ d'applications de transitions
$$
\phi_i: U_i\longrightarrow \RR^n
$$
telles que $\phi_i(U_i)$ est définissable dans $\mathcal{S}$ et telles que les applications de changement de coordonnées 
$$
\phi_{i,j}:=\phi_i  \phi_j^{-1} :  \phi_j(U_i\cap U_j) \longrightarrow \phi_i(U_i\cap U_j)
$$
soient définissables dans $\mathcal{S}$. Le Théorème 1.1 de \cite{BKT} montre que tout quotient arithmétique
$$
S_{\Gamma, G,M}= \Gamma \backslash G/M
$$
 admet une structure de variété $\RR^{alg}$-définissable. On dispose donc d'un atlas fini de   $S_{\Gamma, G,M}$ tel que les applications de changement de coordonnées sont semi-algébriques. On dispose alors pour toute structure $o$-minimale $\mathcal{S}$ d'une structure de variété  $\mathcal{S}$-définissable sur $S_{\Gamma, G,M}$ 
 étendant la structure $\RR^{alg}$. Explicitement un sous-ensemble $V$ de $S_{\Gamma, G,M}$ est définissable dans $\mathcal{S}$ si pour tout $i\in I$, $\phi_i (U_i\cap V)$ est définissable dans $\mathcal{S}$ et si pour tout $(i,j)\in I^2$, la restriction de $\phi_{i,j}$ à
$ \phi_j(U_i\cap U_j\cap V) $ est définissable dans $\mathcal{S}$.

Par ailleurs, Bakker, Klingler et Tsimerman (\cite{BKT} appendix A) montrent que toute variété analytique réel compact à coin est munie naturellement d'une structure de variété $\RR^{an}$-définissable. La compactification de Borel-Serre $S_{\Gamma, G,M}^{BS}$ de $S_{\Gamma, G,M}$ fournit ainsi une structure  de variété $\RR^{an}$-définissable sur $S_{\Gamma, G,M}^{BS}$ et par restriction une structure  de variété $\RR^{an}$-définissable sur $S_{\Gamma, G,M}$.   Le Théorème 1.1 de \cite{BKT} assure alors que cette structure coïncide avec la structure $\RR^{an}$ qui étend la structure de $\RR^{alg}$- variété sur $S_{\Gamma, G,M}$. Notons encore une fois que ces structures dépendent fortement du choix d'un ensemble de Siegel de $G$. Différents choix de compacts maximaux de $G$ donnent lieu à des ensembles de Siegel différents qui à leur tour définissent des structures $\R^{alg}$ et $\R^{an}$ différentes sur $S$. Les méthodes de ce texte sont insensibles à ces questions.

\subsection{Mesures de Haar et dimension} Le résultat simple suivant nous servira dans la suite.
\begin{proposition}\label{Haardim}
Soit~$G$ un groupe algébrique défini sur~$\R$ réel, supposé affine, et 
\[
G(\R)\subseteq\R^n
\]
un plongement (algébrique réel) affine fermé. Soit~$\nu$ une mesure de Haar sur~$G(\R)$ et notons~$d=\dim(\G)$ sa dimension de variété algébrique.

Alors pour toute partie~$A\subseteq \R^n$ définissable dans une structure o-minimale, la partie~$A\cap G$ est borélienne. En outre et~$\nu(A\cap \G(\R))=0$ dès que sa dimension comme partie définissable vérifie~$\dim(A)<d$.
\end{proposition}
N.B.: La notion de dimension utilisée n'importe guère. La dimension de la variété algébrique~$G$ (de Krull, ou au  dessus de sur~$\R$) est aussi celle de~$G(\R)$ comme groupe de Lie, et de~$G(\R)\subseteq \R^n$ comme partie définissable. La preuve, laissée au lecteur repose sur le fait que localement la mesure de Haar  est à densité par rapport à la mesure de Lebesgue et de propriétés simples des ensembles définissables dans une théorie o-minimale.

\section{Familles définissables et convergence de Hausdorff} \label{sec3}

Soit $K$ une partie compacte de $\R^n$. Soit $\mathcal{K}(K)$ l'ensemble des sous-ensembles fermés non vides (donc compacts) de $K$. Alors $\mathcal{K}(K)$ est muni de la distance de Hausdorff qui à  deux fermés $F_1$ et $F_2$ non vides de $K$ associe
$$
\delta_H(F_1,F_2)=\max\{\sup_{y\in F_2}d(y, F_1), \sup_{x\in F_1} d(x,F_2)\}.
$$
L'ensemble $\mathcal{K}(K)$ muni de la distance de Hausdorff $\delta_H$ est séquentiellement compact. Ceci donne un sens à la notion de limite de Hausdorff d'une suite de fermés non vides de $K$. 


Nous reprenons une convention de~\cite{Lion} qui nous permettra de se passer d'hypothèse de fermeture sur les familles définissables considérées dans nos énoncés. Cette convention étend la convergence
au sens de Hausdorff, pour toute suite~$(A_i)_{i\in\N}$ de parties de~$K$, en posant
\[
\lim_{i\in\N} A_i=L\qquad\text{ si et seulement si }\qquad\lim_{i\in\N} \overline{A_i}=L \text{ dans }\mathcal{K}(K).
\]

\subsection{Énoncé de semi-continuité}
Le point de départ de notre approche est l'énoncé suivant, dont nous avons, finalement, trouvé une référence, dans~\cite[Th.~3.1 (1), (2)]{Actes Dries}. L'article~\cite{Lion} donne une approche du point central menant à cet énoncé. Nous recommandons fortement son introduction: pour son explication du résultat, pour les références données, pour la discussion sur l'origine contemporaine du résultat. Mentionnons également~\cite[\S1, Prop.~1]{Zell}, un argument de réduction dimensionnelle via le Lemme du chemin, qui nous a permis, dans une version antérieure, de prouver aisément l'inégalité de dimensions~\eqref{semicontinuite de la dimension}.

\begin{proposition}\label{Theoreme o-minimal 1}
Soit~$n\geq 0$ un entier et soit~$K$ une partie compacte de~$\R^n$.

Soit~$(A(b))_{b\in B}$ une famille définissable de sous-ensembles de~$\R^n$ tous contenus dans~$K$.

Nous étudions, pour une suite~$(b_i)_{i\in\N}$ dans~$B$, la suite extraite~$(A(b_i))_{i\in\N}$.

Supposons que nous avons une limite au sens de Hausdorff~$L=\lim_{i\in\N} A(b_i)$.
\begin{enumerate}
\item Alors~$L$ est une partie compacte de~$\R^n$ contenue dans~$K$ (c'est immédiat) et est définissable dans~$\R^n$ (c'est bien connu),
\item et en outre (c'est notre sujet, et moins connu) la dimension ne peut que chuter à la limite, au sens où on a l'inégalité de semi-continuité inférieure
\begin{equation}\label{semicontinuite de la dimension}
\dim(L)\leq \liminf_{i\geq 0} \dim A(b_i).
\end{equation}
\end{enumerate}
\end{proposition}

\section{Mesures et o-minimalité dans un contexte genéral}\label{sec4} Le Théorème \ref{Proposition mesures} est le résultat principal de cette section, et le pivot de ce travail. Si sa formulation est technique, c'est taillé à nos besoin ultérieurs: sa raison d'être est son invocation pour étudier des mesures homogènes rencontrées dans la théorie de Ratner.

Nous donnons auparavant quelques clarifications de rigueur sur les notions de mesures et de convergence de mesures et un énoncé qui compare limite de Hausdorff de la suite des supports d'une suite de mesures convergente avec le support de la mesure limite.

\subsection{Conventions sur la théorie de la mesure}
Nous travaillons avec des mesures sur des espaces topologiques. Comme elles nous ne serviront que pour étudier leur support, nous  incarnerons les mesures comme fonctions d'ensemble, définies sur la tribu borélienne, dénombrablement additives, à valeurs dans~$\R_{\geq0}\cup\{+\infty\}$:
Toutes les mesures seront supposées positives (réelles). En outre toutes les mesures rencontrées seront aussi supposées localement finies.

Soit donnée~$\mu$ une mesure sur un espace topologique~$K$,  comme fonction d'ensemble.
\begin{enumerate}
\item Nous dirons qu'une partie~$N\subseteq K$ est \emph{négligeable (relativement à~$\mu$)} si elle est contenue dans une partie de la tribu borélienne qui est de mesure nulle.
\item Nous dirons que~$\mu$ est \emph{concentrée} dans une partie~$Z\subseteq K$ si la partie complémentaire~$K\smallsetminus Z$ est négligeable.
\item On appelle~\emph{Support} de la mesure~$\mu$ la partie
\[\operatorname{Supp}{}(\mu):=K\smallsetminus\bigcup\left\{U\subseteq K\,\middle|\, U=\mathring{U}, \mu(U)=0\right\}=\bigcap\left\{Z\subseteq K\,\middle|\,Z=\overline{Z}, \mu(K\smallsetminus Z )=0\right\}.\]
\end{enumerate}
\subsection{Remarques}
\begin{itemize}
\item Si~$\mu$ est \emph{finie}, entendre~$\mu(K)<\infty$, alors~$\mu$ est concentrée sur~$Z$ si et seulement si~$\mu(Z)=\mu(K)$. 
\item Si~$Z$ est fermée, alors que~$\mu$ est concentrée sur~$Z$  cela implique
\(\operatorname{Supp}{}(\mu)\subseteq Z\).
\item Si~$\mu$ est intérieurement régulière alors~$\mu$ est concentrée sur son support. (Ce n'est pas le cas de la mesure \cite[\S7.1.3]{Bogachev2} de Dieudonné.)
\end{itemize}

Dans le reste de cette section les mesures seront (de masse) finie. Dans les section suivantes on étudiera plus précisément des mesures de probabilité, mais aussi des mesures infinies associées.

\subsection{Cadre restreint} Pour l'énoncé principal de cette section, nous pourrons nous contenter de mesures finies \emph{concentrées} sur une même partie compacte~$K$ de~$\R^n$.

\subsubsection{Convergence}
Comme~$K$ est compact, la convergence \emph{étroite}, \emph{vague} et \emph{faible} reviennent au même: elles utilisent les mêmes fonctions tests, prises dans l'algèbre~$C(K)$ des fonctions continues~$K\to\R$. On peut réaliser chaque fonction~$f\in C(K)$ comme restriction d'une fonction continue à support compact sur~$\R^n$ (extension de Tietze, version de Brouwer et Lebesgue~\cite{HazeEncy}). 

Nous noterons, utilisant l'intégrale de Lebesgue
\[
\lim_{i\in\N}
\mu_i=\mu_\infty
\text{, ou }
\mu_i\rightharpoonup\mu_\infty
\text{, ou }
\forall f\in C(K), \int f \mu_i\to\int f\mu_\infty
\]
pour
\[
\forall f\in C(K),\lim_{i\in\N}\int f \mu_i=\int f\mu_\infty.
\]

\subsubsection{Point de vue fonctionnel}
Comme~$K$ est métrisable, toutes les mesures de Borel sont régulières, et correspondent univoquement à leur fonctionnelle d'intégration, d'après la version idoine du Théorème de représentation de Riesz: \cite[1.3.3]{Bogachev3} ou~\cite[2.14, cf. 2.15]{Rudin}, cf. \cite[p\,179]{RieszHistoire}.

\subsubsection{Théorème de Prokhorov} Il nous arrivera d'utiliser l'énoncé suivant, qui est une des formes du Théorème de Prokhorov donne. 
Il se déduit des formes plus complètes~\cite[\S8.6, Th.~8.6.2, cf.~Def.~8.6.1]{Bogachev2} ou~\cite[Th. 4.3.2, Def. 1.4.10]{Bogachev3}, en utilisant~$\mathcal{K}_\varepsilon=K$ dans les définitions, et~$\mathcal{M}=\{\mu_i|i\in\N\}$ dans les Théorèmes.

À noter que cet auteur nomme ``\emph{weak convergence}'' \cite[Def. 8.1.1]{Bogachev2} \cite[Def. 2.2.1]{Bogachev3} ce que d'autres appellent \og{}convergence étroite\fg{}. Comme notre espace ambiant~$K$ est compact, aucune confusion n'est à craindre.

\begin{theorem}\label{Prokhorov} Soit~$K$ un espace compact métrisable.

Toute une suite~$(\nu_i)_{i\in\N}$ de mesures (positives) finies~$\nu_i$ dans~$K$, uniformément bornées, au sens où
\begin{equation}\label{H Prokhorov}
\limsup_{i\in\N}\nu_i(K)<+\infty, \text{(soit encore~$\sup_{i\in\N}\nu_i(K)<\infty$)}
\end{equation}
contient une sous-suite infinie convergente.
\end{theorem}

\subsection{Supports de mesures et limites de Hausdorff}

Soit $K$ un espace compact métrisable. Soit $M(K)$ l'espace des mesures de Borel positives finies. L'énoncé suivant, utile dans la suite, est sans doute bien connu, mais faute de référence nous en donnons une preuve que le lecteur pourra omettre dans une première lecture.

\begin{lemme}\label{lemme faible vietoris}

Soit~$(\mu_i)_{i\in \N}$ une suite  faiblement convergente de mesures appartenant à~$M(K)$,
 
 Soit~$(Z_i)_{i\in \N}$ une suite convergente, au sens de Hausdorff, de parties de~$K$, telle que pour tout indice~$i\in I$ la mesure~$\mu_i$ est concentrée sur~$Z_i$.

Si~$\mu_\infty$ est une limite de~$(\mu_i)_{i\in \N}$ dans~$M(K)$, alors la partie fermée limite
$$
Z_\infty=\lim_{i\in \N}Z_i\in \mathcal{K}(K)
$$
 vérifie
\[
\operatorname{Supp}{}(\mu_\infty)\subseteq Z_\infty.
\]
Si~$\mu_\infty=\mu^r_\infty$ est  la limite de~$(\mu_i)_{i\in \N}$ qui est régulière,  alors~$\mu_\infty$ est concentrée dans~$Z_\infty$.
\end{lemme}


\begin{proof}Par notre convention sur la convergence au sens de Hausdorff, on peut supposer que les~$Z_i$ sont fermés. Donc
\[
\operatorname{Supp}{}(\mu_i)\subseteq Z_i.
\]

Soit~$x$ arbitraire dans~$K\smallsetminus Z_\infty$. Nous allons montrer~$x\not\in \operatorname{Supp}{}(\mu_\infty)$.

Pour tout~$z\in Z_\infty$ il existe~$x\in U_z$ et~$z\in V_z$ deux ouverts disjoints. Par compacité de~$Z_\infty$ on extrait~$z_1,\ldots,z_n$
tels que~$V=V_{z_1}\cup\ldots\cup V_{z_n}$ est un voisinage de~$Z_\infty$, qui est alors disjoint du voisinage~$U=U_{z_1}\cap\ldots \cap U_{z_n}$ de~$x$.

Par le Lemme d'Urysohn, il existe une fonction continue~$f\geq0$, nulle sur~$\overline{V}$ et valant~$1$ sur~$\{x\}$. 

Les parties compactes de~$K$ contenues dans~$V$ forment un voisinage de~$Z_\infty\in\mathcal{K}(K)$ pour la topologie de Vietoris. Pour tout $j$ assez grand,
\[
\operatorname{Supp}{}(\mu_j)\subseteq Z_j
\subseteq V.
\]

Pour tout~$\varepsilon>0$, le compact~$W_\varepsilon=\{x\in K|f(x)\geq\varepsilon\}$ est disjoint de~$\overline{V}$, a fortiori de~$Z_j$, et tout autant de~$\operatorname{Supp}{}(\mu_j)$. Tout point~$w$ de~$W_{\varepsilon}$ a un voisinage~$W_{w,j}$ de mesure nulle pour~$\mu_j$. À indice~$j$ fixé, par Borel-Lebesgue, on extrait un sous-recouvrement fini~$W_{w_1,j},\ldots,W_{w_m,j}$ de~$W_\varepsilon$. Il suit
\[
\mu_j(W_\varepsilon)\leq \mu_j(W_{z_1,j})+\ldots+\mu(W_{z_m,j})=0+\ldots+0=0.
\]
Il s'ensuit~\(0\leq\int f\mu_j=\int_{W_\varepsilon}f\mu_j+\int_{K\smallsetminus W_\varepsilon}\leq 0+ \varepsilon\cdot \mu_j(K)\) et,~$\varepsilon$ étant arbitraire, 
\[\int f\mu_j=0.\]
Par convergence faible il suit~$\int f\mu_\infty=0$. Soit~$W$ le voisinage ouvert~$\{x\in X|f(x)>1/2\}$ de~$x$. Par construction de l'intégrale de Lebesgue nous avons
\[
0=\int f\mu_\infty\geq \frac{1}{2}\cdot \mu_\infty(W), \qquad\text{ et donc }\mu_\infty(W)=0.
\]
Comme~$x$ a un voisinage~$W$ négligeable, nous avons par définition~$x\not\in\operatorname{Supp}{}(\mu_\infty)$, ce qui est était recherché.

Nous avons montré
\[
\operatorname{Supp}{}(\mu_\infty) \subset Z_{\infty}
\] 
pour toute limite faible, régulière ou non. Supposons maintenant que~$\mu_\infty$ est la limite régulière. Il s'ensuit que~$\mu_\infty$ est concentrée dans son support, donc dans~$Z_\infty$.
\end{proof}

\subsection{Énoncé} L'énoncé technique suivant contient  du superflu (on peut abstraire l'une ou l'autre des familles considérées) pour mieux servir les besoins de la preuve de~\ref{Charnière}

\begin{theorem}\label{Proposition mesures}
Soit~$n>0$ un entier, soit~$K\subseteq \R^n$ une partie compacte, et soit  $(\nu_i)_{i\in \N}$une suite convergente 
\[
\lim_{i\in\N}
\nu_i=\nu_\infty,
\]
de mesures finies, de limite~$\nu_\infty$ finie, toutes concentrées dans~$K$.

Soit aussi~$d$ tel que toute partie définissable de~$\R^n$ de dimension strictement inférieur à~$d$ est négligeable pour~$\nu_\infty$: pour toute partie~$A\subseteq\R^n$ définissable
\begin{equation}\label{H1}
\dim(A)<d\Rightarrow \nu_\infty(A\cap K)=0.
\end{equation}

Considérons deux familles définissables~$(A(b))_{b\in B}$ et~$(A'(b'))_{b'\in B'}$, dont les membres sont tous dans~$\R^n$, et deux suites~$(b_i)_{i\in\N}$ dans~$B$ et~$(b'_i)_{i\in\N}$ dans~$B'$.

On suppose que 
\begin{itemize}
\item la première famille enveloppe les supports au sens où
\begin{equation}\label{H0}
\forall i\in\N,~\operatorname{Supp}{}(\nu_i)\subseteq A(b_i)
\end{equation}
\item la deuxième est asymptotiquement de mesure non nulle au sens où
\begin{equation}\label{H2}
\liminf_{i\in\N}\nu_i(A'(b'_i))>0.
\end{equation}
\end{itemize}

On a alors la minoration
\[
\liminf_{i\geq 0} \dim (A(b_i)\cap A'(b'_i))\geq d.
\]
\end{theorem}
Ajoutant comme hypothèse l'égalité inverse,  le Lemme~\ref{Lemme dimension interieur} implique le résultat suivant.
\begin{corollaire}\label{Corollaire mesures}
Si de plus
\begin{equation}\label{H3}
d\geq\limsup_{i\geq 0} \dim(A(b_i)),
\end{equation}
alors
\(d=\lim_{i\geq 0} \dim(A(b_i))\).
Pour~$i$ assez grand, l'ensemble~$A'(b'_i)$ 
contient un ouvert non vide de~$A(b_i)$.
\end{corollaire}

La preuve va être de se ramener à la section précédente et à l'inégalité~\ref{semicontinuite de la dimension}.
Il s'agit sinon de manipuler convergence et support de mesures avec la rigueur necessaire.

\begin{proof}
 Nous pouvons supposer~$A'(b'_i)\subseteq A(b_i)$, quitte à remplacer~$A'(b'_i)$ par
 $$
 C(b_i,b'_i)=A(b_i)\cap A'(b'_i)
 $$
  extrait de la famille définissable~$C(b,b')_{(b,b')\in B\times B'}$.  On substitue alors $(b_i,b'_i)$ à~$b'_i$. 

Nous préparons l'usage de la Proposition~\ref{Theoreme o-minimal 1}, en nous ramenant à des familles uniformément.
La partie compacte~$K$ est contenue dans une partie compacte~$K'$ définissable de~$\R^n$, par exemple une boule centrée à l'origine de rayon suffisamment grand. On peut remplacer~$(A(b))_{b\in B}$ par $(A(b)\cap K')_{b\in B}$ et~$(A(b'))_{b'\in B}$ par~$(A(b')\cap K')_{b'\in B}$, sans porter atteinte aux hypothèses.

Raisonnons par l'absurde, en supposant que l'on peut passer à une sous-suite infinie pour laquelle 
\begin{equation}\label{absurde}
\forall i, \dim A'(b'_i)<d.
\end{equation}

Par la compacité séquentielle rappelée au début de la section~\ref{sec3}, quitte à extraire, la suite~$(A'(b'_i))_{i\in \N}$ converge au sens de Hausdorff, vers une limite que nous baptisons~$L:=\lim A'(b'_i)$, et qui est une partie compacte de~$K$.
D'après la Proposition~\ref{Theoreme o-minimal 1}, appliquée au compact~$K'$, cette partie~$L$ est définissable dans~$\R^n$. Elle a donc une dimension.  L'inégalité~\eqref{semicontinuite de la dimension} et l'hypothèse~\eqref{absurde}, donnent
\[
\dim L\leq \liminf_{i\in\N} \dim (A'(b'_i))<d.
\]
Par hypothèse sur~$d$ et~$\nu_\infty$, nous avons
\begin{equation}\label{L negligeable}
\nu_\infty(L)=0.
\end{equation}
Travaillons maintenant avec la suite~$(\nu'_i)_{i\in\N}$ formée des restrictions 
\[
\nu'_i:=\nu_i\restriction_{A'(b'_i)}.
\]
Leurs masses admettent une borne par excès commune
\[
\limsup_{i\in\N} \nu'_i(K)=\limsup_{i\in\N} \nu_i(K\cap A'(b'_i))\leq \limsup_{i\in\N} \nu_i(K)=\nu_\infty(K)<+\infty,
\]
Par le Théorème~\ref{Prokhorov} de Prokhorov, quitte à extraire, la suite $(\nu'_i)_{i\in \N}$ a une limite~$\nu'_\infty$ qui satisfait la domination
\begin{equation} \label{infini domination}
\nu'_\infty:=\lim_{i \in \N}\nu'_i\leq \lim_{i \in \N}\nu_i =\nu_\infty.
\end{equation}

L'hypothèse~\eqref{H2}, implique la minoration
\[
\liminf_{i \in \N}\nu'_i(K)\ge \liminf_{i \in \N}\nu_i( A'(b'_i))>0.
\]
En utilisant  la fonction test~$1\in C(K)$, on obtient
\[\nu'_\infty(K)=\int_K 1~\nu_\infty'=\lim_{i \in \N}\int_K 1~\nu'_i>0.\]

D'après le Lemme~\ref{lemme faible vietoris}, la mesure~$\nu'_\infty$ est concentrée sur~$L\subseteq K$. Remémorant~\eqref{L negligeable} et~\eqref{infini domination}, voici la contradiction
\begin{multline}
\nu'_\infty(L)=
\nu'_\infty(K)= \lim_{i \in \N}\nu'_i(K)= \lim_{i \in \N}\nu_i\!\restriction_{A'(b'_i)}\!{}(K)
\\
=\lim_{i \in \N}\nu_i(A'(b'_i))>0=\nu_\infty(L)\geq
\nu'_\infty(L).\qedhere
\end{multline}
\end{proof}

\section{Dynamique homogène} \label{sec5} Nous voulons appliquer notre énoncé précédent dans une certaine situation ayant trait aux variétés de Shimura, pour laquelle nous pouvons utiliser les Théorèmes de Ratner, et surtout leurs conséquences précisées par Mozes et Shah. Cela va requérir de mettre en \oe uvre  une couche additionnelle de notions.

\subsubsection*{Références} Nous recommandons chaudement ces quelques précieuses références, de grande qualité, sur le sujet des réseaux arithmétiques~\cite{BorelIntro, Raghunathan, Margulis, Platonov}. Les Théorèmes de Ratner (voir \cite{Witte} pour une présentation) nous serviront par l'entremise de~\cite{MoSh}, et seulement dans le cas de réseaux qui sont arithmétiques.

\subsection{Quotients arithmétiques}  \label{def4}
Soit  $\G$ un groupe algébrique semi-simple $\G$ sur $\QQ$. Soit  $\Gamma\subset \G(\QQ)$ un réseau arithmétique. On fixe une représentation 
 fidèle~
 $\rho:\G\to \GL(W)$, dans un espace vectoriel $W$ sur~$\Q$ de dimension finie $N$. On en déduit une inclusion
 $$
 G\subset \GL(W\otimes \RR)\simeq \GL_N(\RR)\subset \RR^{N^2}.
 $$
 Cela donne un sens aux notions d'ensembles bornés et d'ensembles définissables de $G$. Ces notions sont en fait indépendants du choix d'une représentation fidèle.

\subsection{Définissabilité de certaines familles de doubles classes} \label{sec6}

Nous démontrons dans cette partie que certaines familles de doubles classes sont définissables. Cela résulte de techniques classiques dans la théorie de la réduction pour l'action d'un réseau arithmétique $\Gamma$ d'un groupe de Lie $G$. Faute de référence nous donnons des preuves des énoncés suivants qui jouerons un rôle important dans la suite. 

\begin{proposition} 

Soit~$\Gamma\subseteq \mathbf{G}(\Q)$ un réseau d'un groupe de Lie algébrique réel~$\mathbf{G}(\R)$ défini sur~$\Q$. Soit~$H$ la composante neutre d'un sous-groupe de Lie algébrique~$\mathbf{H}(\R)$ défini sur~$\Q$ sans caractère algébrique rationnel.

Pour toutes parties bornées~$U$ et~$V$ de~$G$, il existe un ensemble fini~$F$ de~$\Gamma$ tel que
\[
\forall g\in V,
(\Gamma Hg) \cap U=(F Hg) \cap U.
\]
\end{proposition}

\begin{proof}  Ajournant le choix de la partie finie~$F\subseteq \Gamma$, on raisonne par double inclusion, l'une étant triviale. Il faut in fine démontrer
\[
(\Gamma Hg) \cap U\subseteq (F Hg)
\]
indépendamment de~$g$ décrivant~$V$. Abstrayons~$g$ en substituant~$UV^{-1}$ à~$U$: 
en effet
\[
(\Gamma Hg) \cap U\subseteq (F Hg)\quad \Leftrightarrow\quad (\Gamma H) \cap Ug^{-1}\subseteq (F H)\quad \Leftarrow\quad   (\Gamma H) \cap UV^{-1}\subseteq (F H).
\]

Tout revient à montrer la finitude de
\(
(\Gamma H) \cap U \pmod{H}
\) 
c'est-à-dire de
\[
(\Gamma H/H )\cap (UH/H).
\]
Alors il suffira de choisir pour~$F$ un ensemble fini d'éléments de  $\Gamma$ dont l'image dans $\Gamma H/H$ recouvre $(\Gamma H/H) \cap (UH/H).$


Par un Théorème classique de Chevalley précisé dans \cite[Thm~5.1]{Borel-LAG}, il existe une représentation fidèle de $\G$ dans un $\Q$-espace vectoriel de dimension finie $W$ et un vecteur $w\in W$ tel que le stabilisateur de $w$ dans $G$ est $H$.

Donc~$G/H\simeq G\cdot w$ est plongé continûment (et même polynomialement et homéomorphiquement \cite[5.1 (Arens)]{BHC} \cite[Lem.~6.2]{Borel-LAG}) dans~$W\otimes\R$.  Comme~$\Gamma$ est un réseau arithmétique, il stabilise un réseau (linéaire)~$\Lambda$ de~$W\otimes\R$ contenu dans $W$. On peut donc remplacer~$w$ par un multiple entier non nul de sorte que~$w$ appartienne à~$\Lambda$.
Donc l'orbite~$\Gamma H/H\simeq \Gamma \cdot w$ est contenue dans un réseau~$\Lambda$, qui est un fermé discret.

Par ailleurs~$U\cdot w$ est borné dans~$W$ (car contenu dans le compact~$\overline{U}\cdot w$ image de~$\overline{U}$).

En fin de compte l' intersection
\[
(\Gamma H/H) \cap (UH/H)
\]
est bornée dans~$W$ et discrète, donc bel et bien finie.
%
%
\end{proof}

\begin{corollaire} \label{dc}
Supposons~$U$ et~$V$ semi-algébriques et bornés. La famille
\[
((\Gamma H) \cap g U)_{g\in V}
\]
est alors semi-algébrique.
\end{corollaire}
\begin{proof}

 Dans
\[
((F H) \cap g U)_{g\in V}
\]
l'ensemble~$H$ est algébrique réel, de même que tout translaté~$f\cdot H$, et donc de l'union finie~$F\cdot H$ de tels translatés. 
La famille~$(FH\cap g U)_{g\in G}$ est définissable, car donnée par la propriété
\[
x\cdot g^{-1} \in U\text{ et }x\in FH.
\]
La restriction de la famille à~$V\subseteq G$ est donc définissable.

\end{proof}
Mettons en exergue le cas~$V=\{g\}$.
\begin{corollaire}
Alors, pour tout semi-algébrique borné~$U$ de~$G$ et~$g$ dans~$G$
\[
(\Gamma H g) \cap U
\]
est définissable.
\end{corollaire}

\subsection{Rappels de dynamique homogène}\label{sec7}
\subsubsection{Haut et Bas}
Baptisons l'application quotient
\(
\theta:G\to \Gamma\sous G
\), visualisée verticalement:
\begin{equation}\label{application verticale}
\begin{tikzcd}
G \arrow{d}\\ \Gamma\sous G
\end{tikzcd}
\end{equation}
Nous allons traiter de mesures à deux étages: \flqq{}~en haut~\frqq{} dans~$G$; et \flqq{}~en bas~\frqq{} dans~$\Gamma\sous G$. Notre but sera d'étudier des objets géométriques en bas par des mesures canoniquement associées. Les énoncés précédents vont s'appliquer aux mesures en haut. 

%

\subsubsection{Mesures de probabilité en bas (Théorie de Ratner)} 

\begin{definition}
Un sous-groupe algébrique~$\HH\leq \G$ défini sur $\Q$ est dit de type $ \mathcal{H}$ si son radical $\mathbf{R}_H$ est unipotent et si $\HH/\mathbf{R}_H$ est un produit presque direct de sous-groupes $\Q$-simples dont les points réels ne sont pas compacts.  Le sous-groupe arithmétique~$\Gamma\cap H\subset\HH(\Q)\cap H\subset H$ est alors un réseau. Il y a donc une mesure de probabilité sur~$(\Gamma\cap H)\sous H$ invariante à droite par~$H$. On note $\mu_\HH$ la mesure sur~$\Gamma\sous G$ donnée par l'image directe de cette probabilité invariante par l'application d'identification puis le plongement fermé (\cite[Th.~1.13]{Raghunathan})
\[
(\Gamma\cap H)\sous H \simeq \Gamma \sous \Gamma \cdot H \hookrightarrow \Gamma\sous G.
\]
\end{definition}

\begin{definition}\label{écriture} Une mesure (de probabilité) \emph{de type Ratner}~$\mu$ sur~$\Gamma\sous G$ est une mesure de la forme
\[
\mu=\mu_\HH\cdot g
\]
translatée par~$g$  de la mesure~$\mu_\HH$, avec~$\HH\leq \G$ un sous-groupe algébrique de type~$\mathcal{H}$ et~$g$ dans~$G$.
\end{definition}

\paragraph{Choix d'\flqq{}~écriture\frqq{}}
Si une mesure~$\mu$ de type Ratner peut  s'écrire de la forme
\[
\mu=\mu_{\HH}\cdot g
\]
avec un couple~$(\HH,g)$, ce couple n'est pas uniquement déterminé par~$\mu$. 
\begin{definition}\label{def ecritures}
Pour~$\mu$ une mesure de probabilité de type Ratner sur~$\Gamma\sous G$, nous appellerons \emph{écriture} de~$\mu$ un couple~$(\HH,g)\in\mathcal{H}\times G$ tel que~$\mu=\mu_\HH\cdot g$.
\end{definition}
Le groupe algébrique~$\HH$ n'est déterminé qu'à conjugaison près par les éléments de~$\Gamma$; et~$\HH$ étant choisi,~$g$ n'est déterminé qu'à translation à gauche par~$H$.

En revanche~$\mu$ et~$g$ étant connus, on retrouve~$H$ comme composante neutre du stabilisateur de~$\mu\cdot g^{-1}=\mu_\HH$ pour l'action de~$G$ par translation, et puis~$\HH$ comme adhérence de Zariski de~$H$ dans~$\G$ (sur~$\Q$).\footnote{L'adhérence de Zariski sur~$\R$ sera en fait définie sur~$\Q$.}

\subsubsection{Mesures localement finies en haut}\label{becar}

Il y a une correspondance bijective entre
\begin{itemize}
\item  les mesures en haut~$\nu$ localement finies et invariantes par translation à gauche par des élements de~$\Gamma$, 
\item et les mesures~$\nu$ localement finies en bas.
\end{itemize}
Cette correspondance dépend d'une normalisation, le choix d'une mesure de Haar (positive)~$\nu_\Gamma$ sur~$\Gamma$, et nous choisissons la mesure de comptage
\[
\nu_\Gamma(B)=\# B\in\{0;1;\ldots;+\infty\}\text{ pour }B\subseteq \Gamma,
\]
mais peu nous importe en définitive. Nous noterons cette correspondance à la manière de la référence~\cite{BBKINT}, au moyen des opérations de quotient et relèvement
\begin{equation}\label{correspondance vague}
\nu\mapsto\mu=\nu_\Gamma\sous \nu\qquad \mu\mapsto\nu=\mu^\sharp.
\end{equation}
La notion générale de \emph{mesure quotient} est développée dans~\cite[INT VII \S2 No. 2 Def.~1 VII.31]{BBKINT}, dans la généralité idiosyncratique, dans l'approche de la théorie de la mesure qui leur est propre. Un cas plus particulier qui nous correspond est~\cite[INT VII \S3 Th.~4 VII.52]{BBKINT}. Un traitement plus classique de notre cas se trouve  dans~\cite[Ch.~I]{Raghunathan}.

\paragraph{} La mesure~$\mu$ n'est pas l'image directe~$\theta_\star\nu$, laquelle n'est pas localement finie, sauf à être nulle, et alors~$\nu$ aussi.

\paragraph{} Sauf cas très particuliers,~$\Gamma$ est infini. Alors la mesure~$\nu$ ne sera pas finie, même si~$\mu$ est finie, sauf à être nulle, et alors~$\mu$ aussi.

\paragraph{} Les mesures~$\nu$ et~$\mu$ se correspondent si et seulement si elles sont localement égales, une fois comparées au moyen d'un homéomorphisme local provenant de~$\theta$. Cela signifie que si~$\mu$ et~$\nu$ se correspondent et si~$\theta$ est injective sur un ouvert~$U\subseteq G$ alors
\[
\nu(\theta(U))=\mu(U)
\]  
et réciproquement que si l'identité vaut pour tout tel ouvert~$U$, alors~$\mu$ et~$\nu$ se correspondent.




\subsection[Résultats de Mozes-Shah]{Résultats de Mozes-Shah (basés sur ceux de Ratner)} \label{secMS}

Mettons en place une terminologie pour exposer les informations que nous tirons de~\cite{MoSh}
concernant le comportement typique des suites de mesures de type Ratner dans~$\Gamma\sous G$.
\begin{definition} Nous dirons qu'une suite de mesures de type Ratner~\((\mu_i)_{i\geq 0}\) est \emph{de la forme Mozes-Shah (dans~$G$)} s'il existe des écritures (cf. Def.~\ref{def ecritures})
\[
\mu_i=\mu_{\HH_i}\cdot g_i
\]
où
\begin{itemize}
\item[(M.-S.~1)] la suite~$(g_i)_{i\geq 0}$ est convergente, de limite notée~$g_\infty$;
\item[(M.-S.~2)] parmi les sous-groupes algébriques de~$G$ (sur~$\R$ ou~$\Q$) peu importe) contenant une infinité des~$\HH_i$, il en existe un unique minimal, que nous noterons~$\HH_\infty=\limsup_{i\geq 0} \HH_i$.
\end{itemize}
Nous dirons alors que~$(\HH_i,g_i)_{i\geq 0}$ est \emph{une écriture de type Mozes-Shah} (dans~$G$) de la suite~$(\mu_i)_{i\geq 0}$.

\end{definition}

Nous avons choisi de décomposer les résultats de Mozes-Shah en trois énoncés, pas indépendants dans leur ensemble. Ce premier, très utile pour deviner les mesures limites d'une suite étudiée, est ramené par ces auteurs aux Théorèmes de Ratner. 
\begin{theorem} \label{MS1}
Une suite~$(\mu_i)_{i\in\N}$ de la forme Mozes-Shah a une limite (étroite)~$\mu_\infty=\lim_{i\in\N}\mu_i$ et, pour une écriture de type Mozes-Shah, sa limite admet l'écriture
\[
\mu_\infty=\mu_{\HH_\infty}\cdot g_\infty,
\]
avec~$\HH_\infty$ et~$g_\infty$ tels que dans la définition.
\end{theorem}
Un point central pour nous est que
\begin{equation}\label{dimension monte}
\dim(\HH_\infty)\geq \limsup\dim(\HH_i)\text{ soit }
\dim(H_\infty)\geq \limsup\dim(H_i)
\end{equation}
en utilisant la dimension comme variété algébrique puis comme groupe de Lie (i.e. comme variété différentielle).

Les  énoncés suivants indiquent que les situations de la forme Mozes-Shah, loin d'être particulières, capturent toutes les situations de convergence. 
\begin{theorem}  \label{MS extraite}
(i) Toute suite \emph{tendue} de mesures de type Ratner contient une sous-suite infinie de la forme Mozes-Shah.

(ii) Toute suite \emph{\underline{étroitement} convergente} de mesures de type Ratner est de la forme Mozes-Shah.
\end{theorem}

\begin{theorem}\label{Critère tendue} Une suite de mesures de probabilités de type Ratner~$(\mu_i)_{i\geq0}$ est tendue si et seulement si il existe une suite d'écritures
\[
\mu_i=\mu_{\HH_i}\cdot g_i
\]
où la suite~$(g_i)_{i\geq 0}$ est \emph{bornée} dans~$G$: il existe une partie compacte~$C$ de~$G$ contenant chacun des~$g_i$.
\end{theorem}

\subsection{Dynamique Homogène sur les espaces hermitien symétriques.}

Les énoncés de la partie précédente s'appliquent directement aux sous-variétés homogènes et aux sous-variétés faiblement spéciales des espaces symétriques hermitiens. On fixe $\G$ comme dans la partie précédente. On suppose de plus que $\G$ est semi-simple de type adjoint et que l'espace symétrique $X$ de $G$ est hermitien. Dans cette situation $S=\Gamma\backslash X$ est un espace localement symétrique hermitien.
On fixe un point $x_0$ de $X$ et on dispose d'un morphisme
$$
\beta=\beta_{x_0}: \Gamma\backslash G\longrightarrow S=\Gamma\backslash X
$$
$$
\Gamma g\mapsto \Gamma g.x_0
$$

\begin{definition}
\begin{itemize}
\item Une mesure de type Ratner sur $S$ est l'image direct $\nu=\beta_*\mu$ d'une mesure de type Ratner sur $\Gamma\backslash G$.
\item Le support d'une mesure de type Ratner est un sous-espace homogène de $S$. 
\item On observe que les mesures de probabilités associées aux sous-variétés faiblement spéciales sont de type Ratner.
\item Une suite de mesures de type Ratner sur $S$ est dite de Mozes-Shah, si elle est de la forme $(\nu_n)_{n\in \NN}=(\beta_*(\mu_n)_{n\in \NN})$ pour une suite $(\mu_n)_{n\in \NN}$ de type Mozes-Shah de $\Gamma\backslash G$.
\end{itemize}
\end{definition}

Avec ces définitions, les énoncés des Théorèmes \ref{MS1}, \ref{MS extraite} et \ref{Critère tendue} valent pour $S$.
Nous retiendrons pour les applications que nous avons en vue le résultat suivant.
\begin{theorem}\label{MSS}
Soit $(Z_n)_{n\in \NN}$ une suite de sous-variétés faiblement spéciales de $S$.  On suppose que la suite de mesures associée $(\nu_n)_{n\in \NN}$ est tendue. Il existe alors une sous-suite infinie de type Mozes-Shah qui converge vers une mesure de type Ratner $\nu_{\infty}$. En passant à cette sous-suite, on peut écrire 
$$
\nu_n=\beta_* (\mu_{\HH_i}\cdot g_n) \longrightarrow \nu_{\infty}=\beta_*(\mu_{\HH_{\infty}.g_{\infty}})
$$
pour un sous-groupe algébrique $\HH_{\infty}$ de type $\mathcal{H}$  tel que pour tout $n$ assez grand $\HH_n\subset \HH_{\infty}$ et $g_{\infty}= \lim_{n\rightarrow \infty} g_n$.
 
\end{theorem}

\subsection{Le \emph{"baby case" du Théorème d'Ax-Lindemann hyperbolique.}}  

Le résultat suivant est une conséquence simple du Théorème d'Ax-Lindemann hyperbolique (ALH) démontré dans \cite{KUY} en utilisant le Lemme 4.1 de \cite{PT}. Une preuve rapide qui n'utilise pas toute la force de 
ALH est donnée dans la Proposition 2.6 de \cite{UY5}. 

\begin{proposition}[{\cite{KUY}}] \label{BAL}
Soit~$\HH$ un sous-groupe algébrique de $\G$ de type $\mathcal{H}$. Soit~$x$ un point de~$X$ et soit 
$$
[H\cdot x]:=\{\Gamma\cdot h\cdot x\in \Gamma\sous X~|~h\in H\}\subseteq S
$$
l'image dans~$S$ de l'orbite de~$H$ dans~$X$ passant par~$x$.

Alors une composante irréductible de  l'adhérence de Zariski de~$[H\cdot x]$ dans~$S$ est une sous-variété faiblement spéciale.

\end{proposition}

\subsection{Dynamique homogène en présence de perte de masse.}
\subsubsection{Introduction}
Pour les suites~$(\mu_i)_{i\geq0}$ de mesures de type Ratner dans~$\Gamma\sous G$, un Théorème de Dani-Margulis implique l'alternative suivante:
\begin{itemize}
\item ou bien la suite~$(\mu_i)_{i\in\N}$ contient une sous-suite tendue, donc une sous-suite étroitement convergente (donc une suite de type Mozes-Shah)
\item ou bien la suite~$(\mu_i)_{i\in\N}$ converge faiblement vers~$0$.
\end{itemize}
Le Théorème de Dani-Margulis va un peu plus loin et implique par exemple que, dans le second cas les groupes $\HH_n$ de type ${\mathcal H}$ peuvent être choisi tels qu'une infinité d'entre eux sont contenus dans un même sous-groupe parabolique propre défini sur~$\Q$.

Les travaux de~\cite{DGU1,DGU2} étudient plus précisément ce phénomène en complétant les travaux de Mozès-Shah. Si les derniers utilisent quasi-exclusivement les Théorèmes de Ratner, les premiers utilisent aussi des propriétés fines de la théorie de la réduction associée aux groupes arithmétiques.

\subsubsection{Sous-groupes paraboliques et chambres de Weyl associés} Nous commençons par introduire quelques notations relatives à un sous-groupe parabolique. Une référence standard, complète et détaillée est~\cite[III.1]{BorelJi} dont nous rappelons les quelques constructions qui nous serviront.

On suppose donnés un sous-groupe parabolique~$\PP\leq \G$ défini sur~$\Q$, et un sous-groupe compact maximal~$K_{\infty}\leq \G(\R)$ dont l'involution de Cartan globale est notée~$\Theta:G\to G$.

Le sous-groupe~$L=P\cap\Theta(P)$ est un sous-groupe réductif (défini sur~$\R$) maximal de~$P$, aussi dit  \emph{sous-groupe de Levi}. Si~$\NNN$ désigne le plus grand sous-groupe algébrique unipotent de~$\PP$ (défini sur~$\R$ ou~$\Q$, cela revient au même), aussi dit \emph{radical unipotent}, alors on a la \emph{décomposition de Lévi} (définie sur~$\R$)
$$
P=N\cdot L,
$$
et $N\cap L=\{e\}$. 

Le quotient~$\mathbf{Q}=\PP/\NNN$ est un groupe algébrique défini sur~$\Q$. C'est le plus grand quotient réductif de~$\PP$, parfois dit \emph{(groupe de) Levi quotient}.
Le centre de~$\mathbf{Q}$ contient un plus grand sous-tore~$\mathbf{S}_Q\leq \mathbf{Q}$ défini et déployé sur~$\Q$. Comme on a un isomorphisme, algébrique sur~$\R$
$$
L_\R\to Q_\R 
$$
$\mathbf{S}_Q$ se relève en un tore~$A\leq L$ défini sur~$\R$, et déployé sur~$\R$ (pas forcément maximal) et on a une décomposition presque directe

$$
L=MA.
$$
De plus le groupe $H_P= NM$ provient d'un $\QQ$-sous-groupe algébrique $ \mathbf{H}_P$ de $\G$. Dans cette situation $\Gamma_P=\Gamma \cap \mathbf{H}_P$ est un réseau de 
$\mathbf{H}_P$.

On dispose alors  des décompositions de Langlands associées
\begin{equation}\label{eqLang}
P=NMA \ \ \ \ \ \ \    G=PK_{\infty}=NMAK_{\infty} \ \ \ \ \ \ \ \    X=NX_PA
\end{equation}
où $X_P=M/(K_{\infty}\cap M)$.
On écrira $g=n_gm_ga_gk_g$ (ou simplement $g=nmak$ si cela ne prête pas à confusion)  la décomposition de Langlands d'un élément $g\in G$ relativement à $\PP$ et à $K_{\infty}$.

 On note
\[
A^+\leq A(\R)
\]
la composante neutre de $A(\R)$ au sens des groupes de Lie. La donnée du sous-groupe parabolique~$P$ contenant~$A$ permet de définir une chambre de Weyl positive
comme suit. La représentation adjointe \underline{\em à droite} de~$P$ sur~$\mathfrak{g}$ 
\[
X\cdot p = p^{-1} X p.
\]
stabilise la sous-algèbre de Lie~$\mathfrak{n}$ associée à~$N$.
L'action de~$A$ est simultanément diagonalisable sur~$\R$ et les espaces propres simultanés sont de la forme~$\mathfrak{n}_\alpha$ associés à des caractères algébriques réels
\[
\alpha:A\to GL(1)_\R
\]
par
\[
\mathfrak{n}_\alpha=\{X\in\mathfrak{n}|\forall a\in A(\R), a\cdot X=\alpha(a)\cdot X\}.
\]
Nous noterons
\[
\Phi(A,P)
\]
l'ensemble des caractères de~$A$ tels que~$\mathfrak{n}_\alpha\neq \{0\}$. Nous considérerons la chambre positive
\[
A^+_{\geq 0}=\{a\in A^+|\forall \alpha\in\Phi(A,P), \alpha(a)\ge 1\}
\]
et pour~$t\in \R^+$ tel que $t\ge 1$, on note
\[
A^+_{\ge t }= \{a\in A^+|\forall \alpha\in\Phi(A,P), \alpha(a')\ge t\}.
\]

\subsubsection{Ensembles de Siegel}
Soit $\PP\leq \G$ un sous-parabolique  défini sur~$\Q$ et~$K_{\infty}\leq \G(\R)$ un sous-groupe compact maximal, nous définissons une classe de sous-ensembles de~$G$.
On écrit suivant la décomposition de Langlands (\ref{eqLang})
$$
P=NMA=H_PA.
$$
 Soit $V$ un ouvert borné  semi-algébrique de $H_P=NM$ et $t\ge  1$.

L'ensemble 
$$
\Sigma(\PP,V,t,K_{\infty}):=VA^+_{\ge t}K_{\infty}
$$
est appelé ensemble de Siegel dans $G$  relativement à $\PP$. Si $M$ est un sous-groupe compact de $K_{\infty}$ et $\alpha_M:G\longrightarrow G/M$ est l'application quotient alors 
$$
\Sigma_M(\PP,V,t,K_{\infty}):=\alpha_M(\Sigma(\PP,V,t,K_{\infty}))
$$
est appelé ensemble de Siegel dans $G/M$ relativement à $\PP$. Notons que la condition de semi-algébricité sur $V$ que nous imposons n'est pas usuellement requise dans la littérature, mais nous ne considérons que des ensembles de Siegel avec cette condition sur $V$ dans la suite.

 La théorie de la réduction est résumée dans la Proposition suivante [\cite{BorelJi} , Proposition III.2.19, p. 284] 
\begin{proposition}\label{Siegel}
Soit $K_{\infty}$ un compact maximal fixé de $G$. Soit $(\PP_1,\dots,\PP_r)$ un ensemble de représentants des $\Gamma$-classes de conjugaison de $\QQ$-sous groupes paraboliques de $\G$.
\begin{itemize}
\item (i) Il existe des ensembles de Siegel
$\Sigma(\PP_i,V_i,t_i, K_{\infty}):=V_iA^+_{i,\ge t}K_{\infty}$ de $G$ associés aux $\PP_i$ dont la réunion des images dans $S_{\Gamma,G,M}=\Gamma\backslash G/M$ recouvrent $S_{\Gamma,G,M}$.
\item (ii) Soient $\mathbf{Q}_1$ et $\mathbf{Q}_2$ deux $\QQ$-paraboliques  de $\G$ et $\Sigma(\mathbf{Q}_i,V_i,t_i,K_{\infty})$ des ensembles de Siegel associés. Alors l'ensemble
$$
\{\gamma\in \Gamma \ \vert \gamma \Sigma(\mathbf{Q}_1,V_1,t_1,K_{\infty})\cap \Sigma(\mathbf{Q}_2,V_2,t_2,K_{\infty}) \neq \emptyset\}
$$
est fini. Si de plus $\mathbf{Q}_2$
 n'est  pas $\Gamma$-conjugué à $\mathbf{Q}_1$ et $V_1$ et $V_2$ sont fixés, 
 $$
 \gamma \Sigma(\mathbf{Q}_1,V_1,t_1, K_{\infty})\cap \Sigma(\mathbf{Q}_2,V_2,t_2,K_{\infty}) =\emptyset
 $$
  pour $t_1$ et $t_2$ suffisamment grand.
 \item (iii) Soit $\PP$ un $\QQ$-parabolique  de $\G$ et $\Sigma(\mathbf{P},V,t)$ un ensemble de Siegel  de $\G$ associé à $\PP$. Alors pour $t$ suffisamment grand 
 $$
 \gamma \Sigma(\PP,V,t,K_{\infty})\cap \Sigma(\PP,V,t)=\emptyset
 $$
 pour tout $\gamma\notin \Gamma- \Gamma_P$, où l'on a noté $\Gamma_P=\Gamma\cap \PP$.
 \end{itemize}
\end{proposition}

\subsubsection{Propriétés de définassibilité des applications d'uniformisation des quotients arithmétiques.}\label{sDefiUnif}

Nous aurons besoin dans la suite d'énoncés assurant que la restriction de l'application d'uniformisation d'un quotient arithmétique à un domaine de Siegel est définissable dans une théorie o-minimal convenable. Le cas de l'espace des modules ${\mathcal A}_g$ des variétés abéliennes principalement polarisées de dimension $g$ est traité par Peterzil et Starchenko \cite{PT}. Le cas général d'une variété de Shimura, ou d'espace localement symétrique hermitien arithmétique est traité dans \cite{KUY}. Le cas des quotients arithmétiques arbitraires est traité par Bakker, Klingler et Tsimerman \cite{BKT}. On résumé cet ensemble de résultats dans l'énoncé suivant.

Soit $\G$ un groupe semi-simple connexe sur $\Q$, $\Gamma \subset \G(\Q)^+$ un réseau sans torsion et $M\subset G$ un sous-groupe compact. On rappelle que le quotient arithmétique 
$S_{\Gamma,G,M}=\Gamma\backslash G/M$ est muni d'une  structure de variété $\RR^{alg}$-définissable.   Si $G$ est de type hermitien, $M=K_{\infty}$, alors $X=G/K_{\infty}$ est un espace hermitien symétrique et par Baily-Borel $S_{\Gamma,G,M}$ a une unique structure de variété algébrique quasi-projective. Les deux structures de variétés $\RR^{alg}$-définissables ainsi obtenues sont différentes mais d'après \cite{BKT} les extension de ses deux structures dans $\RR^{an, exp}$ coïncident. On note comme précédemment $\alpha_M:G\longrightarrow G/M$ l'application quotient. Quand $M=K_{\infty}$ correspond à un point $x_0$ de $X=G/K_{\infty}$ on utilise aussi la notation $\alpha_{K_{\infty}}= \alpha_{x_0}$.

\begin{proposition}\label{defpi}
(i) Soit $\pi: G/M \longrightarrow S_{\Gamma,G,M}=\Gamma\backslash G/M$, soit 
$$
\Sigma_M=\Sigma_M(\PP,V,t,K_{\infty})=\alpha_M (\Sigma(\PP,V,t,K_{\infty}))
$$
 un ensemble de Siegel de $G/M$ associé à un sous-groupe parabolique $\PP$ et à $K_{\infty}$ contenant $M$. 
Alors l'application 
$$
\pi_{\vert \Sigma_M}: \Sigma_M\longrightarrow S_{\Gamma,G,M}
$$
est définissable dans $\RR^{alg}$. Elle est donc définissable dans toute théorie o-minimale étendant la structure de variété $\RR^{alg}$-définissable sur $ \Sigma_M\longrightarrow S_{\Gamma,G,M}$. En particulier 
$$
{\pi\circ \alpha_M}_{\vert \Sigma(\PP,V,t,K_{\infty})}: \Sigma(\PP,V,t,K_{\infty})\longrightarrow  S_{\Gamma,G,M}
$$
est définissable dans $\RR^{an, exp}$

(ii) Si $S_{\Gamma,G,K_{\infty}}$ est localement symétrique hermitien
Alors l'application 
$$
\pi_{\vert \Sigma_{K_{\infty}}}: \Sigma_{K_{\infty}}\longrightarrow S_{\Gamma,G,K_{\infty}}
$$
est définissable dans $\RR^{an, exp}$. De même 
$$
{\pi\circ\alpha_{x_o}}_{\vert \Sigma(\PP,V,t,K_{\infty})}: \Sigma (\PP,V,t,K_{\infty}) \longrightarrow S_{\Gamma,G,K_{\infty}}
$$
est définissable dans $\RR^{an, exp}$. 

\end{proposition}

\begin{remarque}
Quand $M=K_{\infty}$, $X=G/K_{\infty}$ est un espace symétrique (hermitien ou non), la structure de variété $\R^{alg}$-définissable sur $S=\Gamma\backslash X$ ne dépend pas du choix de $K_{\infty}$. Soit $\PP$ un $\Q$-parabolique. Soient $K_{\infty}$ et $K'_{\infty}$ deux compacts maximaux de $G$.
On note $x_0$ et $x_0'$ les points  de $X$ correspondant à $K_{\infty}$ et à $K'_{\infty}$. Pour $x\in X$, on note 
$$
\alpha_{x}: G\longrightarrow X \mbox{ l'application } g\mapsto g.x.
$$
 Comme $G=PK_{\infty}$, il existe $p\in P$ tel que $x'_0=p.x_0$ et $\alpha_{p.x_0}=\alpha_{x_0}\psi_p$ où $\psi_p$ désigne l'application de translation à droite par $p$ sur $X$. Dans cette situation
 $$
 \psi_p: \Sigma(\PP,V p^{-1},t, K'_{\infty} ))\longrightarrow \Sigma(\PP,V,t, K_{\infty})
 $$ 
 est un isomorphisme semi-algébrique qui induit l'isomorphisme de $\R^{alg}$-structure définissable sur 
 $$
 \alpha_{x_0}(\Sigma(\PP,V,t, K_{\infty}))= \alpha_{p.x_0} (\Sigma(\PP,V p^{-1},t, K'_{\infty} ))\subset S=\Gamma\backslash X.
 $$
 On construit l'isomorphisme de $\R^{alg}$-structures définissables sur $S$ en recouvrant $S$ par des Siegel convenables relativement  à un système $(\PP_1,\PP_2,\dots, \PP_r)$ de représentants des $\Q$-paraboliques de $\G$  modulo la conjugaison par $\Gamma$ comme dans la Proposition \ref{Siegel}. 
 
 Quand $M$ n'est pas maximal, les $\R^{alg}$-structures définissables sur $S_{\Gamma, G,M}$ associées à deux compacts maximaux distincts peuvent différer de manière essentielle. On peut voir que déjà pour $\G=\SL_2$ il n'est pas en général possible de construire des isomorphismes de $\R^{an,exp}$-structures entre les extensions   $\R^{an,exp}$ des $\R^{alg}$-structures définissables sur $\SL(2,\Z)\backslash \SL(2,\R)$ associées à des compacts maximaux de $G$ distincts.  Les résultats que nous avons en vue sont valables pour tout les choix de $\R^{alg}$-structures définissables sur $S_{\Gamma, G,M}$ et les choix en questions ne jouerons aucuns rôles dans la suite de ce texte.
\end{remarque}

\subsubsection{Suite de mesures de type DGU}\label{sDGU}

\begin{definition}\label{defDGU}
Nous dirons qu'une suite de mesures de type Ratner~\((\mu_i)_{i\in \N}\) est \emph{de la forme DGU (dans~$G$, muni de~$K_{\infty}$, relativement à~$\PP$)} s'il existe des écritures
dans la décomposition de Langlands~(\ref{eqLang})
\begin{equation}\label{def ecritures DGU}
\mu_i=\mu_{\HH_i}\cdot g_i\qquad{g_i=n_i\cdot m_i \cdot a_i\cdot k_i} \qquad{h_i=n_im_i\in H_P}
\end{equation}
où
\begin{itemize}
\item[{[DGU~1]}] la suite~\((\widehat{\mu_i})_{i\in \N}:=(\mu_{\HH_i}\cdot h_i)_{i\geq 0}\) est de la forme Mozès-Shah dans $\Gamma_P\backslash H_P\subset \Gamma \backslash G$, soit
\begin{itemize}
\item[(M.-S.~1)] la suite~$(h_i)_{i\in \N}$ est convergente, de limite notée~$h_\infty$;
\item[(M.-S.~2)] parmi les sous-groupes algébriques de~$\G$  contenant une infinité des~$\HH_i$, il en existe un unique minimal, que nous noterons~$\HH_\infty=\limsup_{i\in \N} \HH_i$.
\item On a de plus $\HH_i\subset \HH_P$ pour tout $i\ge 0$.
\end{itemize}
\item[{[DGU~2]}] la suite~$(a_i)_{i\in \N}$ vérifie
\[
\forall \alpha\in\Phi(P,A_P), \lim_{i\to \infty}\alpha(a_i)=\infty
\]
\end{itemize}
Nous dirons alors que~$(\HH_i,h_i,k_i)_{i\in \N}$ est \emph{une écriture de type DGU} (dans~$G$, muni de~$K_{\infty}$, relativement à~$P$) de la suite~$(\mu_i)_{i\in \N}$.
\end{definition}
Un des points centraux du travail~\cite{DGU2} est la généralisation suivante du Théorème~\ref{MS extraite} qui résume la preuve du Théorème 1.1 de \cite{DGU2} qui s'obtient en vérifiant le critère du Théorème 4.2 de \cite{DGU2}.

\begin{theorem}
Soit~$(\mu_i)_{i\in \N}$ une suite de type Ratner dans~$\Gamma\sous G$.

Il existe un sous-groupe parabolique~$\PP$ tel que pour tout sous-groupe compact maximal $K_{\infty}\subset \G(\R)$,
il existe une sous-suite extraite infinie de la forme DGU (dans~$G$, muni de~$K_{\infty}$, relativement à~$\PP$).
\end{theorem}

Nous aurons besoin de la conséquence simple suivante de ces énoncés.
\begin{lemme}\label{bonK}
Soit $\epsilon>0$. Soit $\theta: G\longrightarrow \Gamma \backslash G$ et soit $(\mu_i)_{i\in \NN}$ une suite de type Ratner dans $\Gamma\backslash G$ avec une écriture de type DGU comme dans la définition \ref{defDGU}, dans $G$  muni de $K_{\infty}$ relativement
à $\PP$. Il existe un ensemble de Siegel $\Sigma(\PP, V_P,t,K_{\infty})=V_PA^+_{\ge t}K_{\infty}$ dans $G$ relativement à $\PP$ tel que pour tout $i$ assez grand
$$
\widehat{\mu}_i(\theta (V_P))= {\mu}_i(\theta (V_Pa_i))\ge 1-\epsilon.
$$
\end{lemme}
Comme la suite  $ (\widehat{\mu_i})_{i\geq 0}:=(\mu_{\HH_i}\cdot h_i)_{i\geq 0}$ est de la forme Mozès-Shah dans $\Gamma_P\backslash H_P$, on sait que $\widehat{\mu_i}$ converge faiblement vers 
$\widehat{\mu}_{\infty}=\widehat{\mu}_{H_{\infty}} \cdot h_{\infty}$ dans $\Gamma_P\backslash H_P$. On peut alors trouver un compact $C$ de $\Gamma_P\backslash H_P$ tel que 
$\widehat{\mu}_{\infty}(C)\ge 1-\frac{\epsilon}{2}$. Pour tout $i$ assez grand on a $\vert \widehat{\mu}_i (C)-\widehat{\mu}_{\infty}(C)\vert \leq  \frac{\epsilon}{2}$. On peut alors choisir un compact semi-algébrique $V_P$ de $H_P$ tel que $C\subset \theta(V_P)$ qui aura les propriétés voulues.  On remarquera aussi que pour $i$ assez grand $a_i\in A^+_{\ge t}$ d'après la condition [DGU2].

\section{Mesures homogènes et o-minimalité dans les espaces localement homogènes}  \label{sec8} Voici l'énoncé charnière de notre méthode: il condense ce qui nous servira dans la preuve du Théorème principal dans le cas tendu. Il combine nos énoncés à propos de la dimension de familles définissables, avec les propriétés sur les mesures de type Ratner.

Une variante à un paramètre de cet énoncé est donnée à la section \ref{v1p}. Elle nous sera utile pour la démonstration du Théorème principal quand on ne se place plus à priori dans le cas tendu.

\begin{theorem}[]\label{Charnière}  

Soit $(\mu_i)_{i\in \N}$ une suite de mesures de type Ratner, de la forme Mozes-Shah, dans une écriture de type Mozes-Shah, 
\begin{equation*}
\mu_i=\mu_{\HH_i}\cdot g_i\longrightarrow\mu_\infty=\mu_{\HH_\infty}\cdot g_\infty. 
\end{equation*}

Soit aussi~$\wV\subseteq \Gamma\sous G$ une partie mesurable et~$U\subseteq G$ une partie définissable bornée. Nous faisons  deux hypothèses.
\begin{itemize}
\item \emph{Hypothèse en bas:} 
\begin{equation}\label{Inégalité en bas}
\liminf_{i\in\N}\mu_i(\wV\cap\theta(U))>0
\end{equation}
\item \emph{Hypothèse en haut:} La partie~$\theta^{-1}(\wV)\cap U$ est définissable.
\end{itemize}

Alors, en haut, sauf pour au plus un nombre fini d'indices~$i\in\N$, l'ensemble
\(\theta^{-1}(\wV)\cap U\) contient un ouvert non vide de $\Gamma \cdot H_\infty\cdot g_i$.

Pour ces mêmes indices, en bas, l'ensemble
\(\widehat{V}\cap \theta(U)\) contient un ouvert non vide de \linebreak $\Gamma\sous\Gamma \cdot H_\infty\cdot g_i$.
\end{theorem}

Le point frappant à réaliser dans cet énoncé est que l'on suppose a priori que~$\wV$ contient une partie de l'ensemble~$\Gamma\sous\Gamma \cdot H_i\cdot g_i=\operatorname{Supp}{}(\mu_i)$, mesurée par~$\mu_i$ mais que  l'on déduit a posteriori que~$\wV$ contient une partie ouverte de l'ensemble plus grand~$\Gamma\sous\Gamma \cdot H_\infty\cdot g_i$. Il est pourtant des cas où~$H_\infty$ soit strictement plus grand que les~$H_i$. 

Le passage de la première conclusion à la seconde n'est pas surprenant et s'obtient de la manière suivante.
\begin{proof} 

Soit $Z=\Gamma\sous\Gamma \cdot H_\infty\cdot g_i$, pour un indice~$i\in\N$ assez grand.
D'après la conclusion de l'énoncé principal pour un tel $i$, il existe un ouvert non vide~$W$ de~$\theta^{-1}(Z)=\Gamma \cdot H_\infty\cdot g_i$ contenu dans~$U\cap\theta^{-1}(\wV)$. Comme $\theta$ est un homéomorphisme local,  c'est une application ouverte et la restriction de $\theta$ à$\theta^{-1}(Z)\rightarrow Z$ est ouverte.
On en déduit que  $\theta(W)$ est un ouvert non vide de $Z$ qui est contenu dans $\theta(U)\cap \wV$.

\end{proof} 

\begin{proof}{~} Nous démontrons maintenant la première phrase de conclusion.
\subsubsection{Instanciation} 
Nous spécifions une instance du Théorème~\ref{Proposition mesures}. 

Comme dans la section \ref{def4}, nous posons~$n=N^2$  et nous réalisons $G$ comme sous-ensemble de  $\R^{N^2}=\R^n$. On définit 
\[K=\overline{U}\subseteq G\subseteq\R^n.\]
Alors~$K$ est une partie compacte car~$U$ est borné et $K$ est définissable comme clôture d'un ensemble définissable.


Soit aussi une partie définissable bornée~$W$ de~$G$ qui contient~$\{g_i|i\in\N\}$.

Comme~$K$ et~$W$ sont définissables et bornés, nous pouvons  appliquer le Corrolaire~\ref{dc}, selon lequel la famille
\begin{equation}\label{invoque définissabilité}
(A(g))_{g\in W}=\left(K\cap (\Gamma H_\infty g)\right)_{g\in W}
\end{equation}
est une famille définissable.

Soit $B'=\{b'\}$ un ensemble ayant un unique élément. Nous utilisons la famille constante
\begin{equation}\label{invoque définissabilité 2}
A'(b')_{b'\in B'}=\left(\theta^{-1}(\wV)\cap U\right)_{b'\in B'}
\end{equation}
qui est une famille  définissable par l'hypothèse en haut. Nous choisissons alors la suite constante
\[
(b'_i)_{i\geq0}=(b')_{i\geq0}.
\]

Soit ${\mu_i}^\sharp$ la mesure localement finie sur $G$ associé 
à la mesure $\mu_i$ en bas, par le procédé détaillé dans la section \ref{becar}. On définit alors $\nu_i:={{\mu_i}^\sharp}\restriction_K$, qui est une mesure finie sur $K$.

\begin{lemme}
Quitte à extraire une sous-suite infinie, on peut supposer que la  suite $(\nu_i)_{i\in \N}$  converge faiblement vers une mesure $\nu_{\infty}$. De plus 
\begin{equation}\label{domination2}
\nu_\infty\leq {\mu_\infty}^\sharp\restriction_K.
\end{equation}

\end{lemme}
\begin{proof}

\begin{equation}\label{une domination}
\limsup_{i\in\N} \nu_i(K)=\limsup_{i\in\N} \mu^\sharp_i\restriction_K(K)\leq {\mu_\infty}^\sharp(K).
\end{equation}
Le membre de droite est fini car~$K$ est borné et~${\mu_\infty}^\sharp$ localement finie. Cela établi l'hypothèse~\eqref{H Prokhorov} du Théorème~\ref{Prokhorov} de Prokhorov qui assure la conclusion du Lemme.
\end{proof}

Nous avons en main tous les objets qui jouent un rôle dans l'énoncé du Théorème~\ref{Proposition mesures}. On pose  

\begin{itemize}
\item $(A(b))_{b\in B}=(A(g))_{g\in W},$
\item $ (b_i)_{i\in\N}=(g_i)_{i\in\N},$
\item $ (\nu_i)_{i\in\N}=({{\mu_i}^\sharp}\restriction_K)_{i\in\N},$
\item $ d=\dim(H_\infty),$
\item $\nu_{\infty}=\lim (\nu_i)$.
\end{itemize}

La conclusion recherchée n'est autre que la conclusion du Corollaire~\ref{Corollaire mesures}; il ne nous reste donc plus qu'à vérifier les hypothèses: \eqref{H1}, \eqref{H0}, \eqref{H2}, \eqref{H3}.
\subsubsection{Vérifications}

Nous avons
\[
\operatorname{Supp}{}\nu_i
=
\operatorname{Supp}{}{\mu_i}^\sharp\restriction_K 
\subseteq
K\cap \operatorname{Supp}{}({\mu_i}^\sharp) 
=
K\cap \Gamma\cdot H_i\cdot g_i
\subseteq
K\cap \Gamma\cdot H_\infty\cdot g_i=A(g_i)
\]
où la dernière inclusion vaut sauf pour peut-être un nombre fini d'indices~$i$, d'après les propriétés des suites de mesures de types Mozes-Shah données dans la section \ref{secMS}.
Cela assure la validité de~\eqref{H0}.

D'après la Proposition \ref{Haardim} tout ensemble définissable~\(A\subseteq \R^n\) tel que
$$
\dim(A)<\dim(H_\infty)=d
$$
 est négligeable pour pour~$\mu_\infty^\sharp$. Alors~$A\cap K$ sera a fortiori négligeable pour~${\mu_\infty}^\sharp\restriction_K$ et, vu la domination~\eqref{domination2}, pour~$\nu_\infty$. On a établi l'hypothèse~\eqref{H1}.

L'hypothèse~\eqref{Inégalité en bas} en bas, n'est qu'une manière d'écrire en bas l'hypothèse~\eqref{H2}.
Nous ne démontrons que l'implication utilisée. Nous passons de l'une à l'autre des inégalités par
\begin{equation}\label{eq7.6}
\nu_i(A'(b'_i))=\int_{K} \,\,\,f\restriction_K~{{{\mu_i}^\sharp}\restriction_K}=
\int_{G}\,\,\, f~{\mu_i}^\sharp
=\int_{\Gamma\sous G} f^\flat \mu_i \geq\int_{\Gamma\sous G} h \mu_i=\mu_i(\wV\cap \theta(U))
\end{equation}
où l'on a choisi
$$
f:=\chi_{\theta^{-1}(\wV)\cap U}:K\to \R_{\geq0}  \mbox{ , } h:=\chi_{\wV\cap\theta(U)}:\Gamma\backslash G\to \R_{\geq0}
$$
 et 
 $$
 f^\flat(\Gamma\cdot g):=\#\left(U\cap\theta^{-1}(\wV)\cap \theta^{-1}(\{\Gamma\cdot g\})\right)\geq \chi_{\wV\cap \theta(U)}(\Gamma\cdot g).
 $$
 On déduit alors \eqref{H2} de l'inégalité en bas ~\eqref{Inégalité en bas} car on a 
\begin{equation}\label{preuve majoration hypothèse en bas}
\liminf \nu_i(A'(b'_i))\geq\liminf \mu_i(\wV\cap\theta(U))>0.
\end{equation}

Pour l'inégalité~\eqref{H3}, nous estimons tout d'abord
\[
\dim A(b_i)\leq \dim (\Gamma\cdot H_{\infty}\cdot g_i) = \dim (H_{\infty})
\]
où la première dimension est la notion définissable et les autres celle de variété différentielle... Alors~\eqref{H3} découle de~\eqref{dimension monte}.

Nous avons bien vérifié les hypothèses voulues. La conclusion  du Corollaire \ref{Corollaire mesures} vaut donc pour notre sous-suite. Ceci termine la preuve du Théorème \ref{Charnière}.  

\end{proof}


\subsection{Variante à paramètres}\label{v1p}
La Proposition~\ref{Charnière} se généralise en la variante à paramètres ci-dessous, dont elle est  le cas particulier~$A=\{e\}$. Cette généralisation nous servira dans l'étude des suites de mesures de type Ratner quand de la masse se perd à l'infini.
\begin{theorem}[]\label{Charnière DGU}  
Soit  $ (\widehat{\mu_i})_{i\in \N}$ une suite de mesures de la forme Mozes-Shah, dans une écriture de type Mozes-Shah,
\begin{equation*}
\widehat{\mu_i}=\mu_{\HH_i}\cdot h_i\longrightarrow \widehat{\mu}_\infty=\mu_{\HH_\infty}\cdot h_\infty,
\end{equation*} 
ainsi qu'une suite~$(a_i)_{i\in\N}$ choisie dans une partie définissable~$A$ de~$G$, et posons enfin
\[
(\mu_i)_{i\in\N}=(\widehat{\mu}_i\cdot a_i)_{i\in\N}.
\]

Soit encore~$\wV\subseteq \Gamma\sous G$ une partie mesurable et~$U\subseteq G$ une partie définissable bornée de~$G$.

Nous faisons ces deux hypothèses.
\begin{itemize}
\item \emph{Hypothèse en bas:} Nous avons
\begin{equation}\label{Inégalité en bas DGU}
\liminf_{i\in\N}\mu_i(\wV\cap\theta(U\cdot a_i))>0,
\end{equation}
\item \emph{Hypothèse en haut:} La partie~$\theta^{-1}(\wV)\cap (U\cdot A)$ est définissable.
\end{itemize}

Alors, en haut, sauf pour au plus un nombre fini d'indices~$i\in\N$, l'ensemble
\(\theta^{-1}(\wV)\cap (U\cdot a_i)\) contient un ouvert non vide de
$\Gamma \cdot H_\infty\cdot h_i\cdot a_i.$

Pour ces mêmes indices, en bas, l'ensemble
\(\wV\cap \theta(U\cdot a_i)\) contient un ouvert non vide de
$$
\Gamma\sous\Gamma \cdot H_\infty\cdot h_i\cdot a_i.
$$
\end{theorem} 
\subsubsection{Démonstration} La démonstration est essentiellement  la même que celle du Théorème~\ref{Charnière}, et nous  n'indiquerons que les modifications à apporter pour prendre en compte le paramètre~$a_i\in A$.
\paragraph{Changement de famille}
La famille~\eqref{invoque définissabilité 2} sera remplacée par
\begin{equation}\label{invoque définissabilité 2 DGU}
A'(b')_{b'\in B'}=\left((\theta^{-1}(\wV)\cap (U\cdot a))\cdot a^{-1}\right)_{a\in A},\text{ pour }B'=A,
\end{equation}
et nous choisissons la suite
\[
(b'_i)_{i\in\N}=(a_i)_{i\in\N}.
\]
Vérifions la définissabilité de cette famille. On rappelle que l'ensemble~$E=\theta^{-1}(\wV)\cap (U\cdot A)$ est définissable.
Pour $a\in A$, on a $U\cdot a =(U\cdot A) \cap (U\cdot a)$ de sorte que   
\[\theta^{-1}(\wV)\cap (U\cdot a)=\theta^{-1}(\wV)\cap (U\cdot A)\cap (U\cdot a)= E\cap (U\cdot a)\]
on peut donc récrire notre famille sous la forme
\[
A'(b')_{b'\in B'}=\left((E\cap (U\cdot a))\cdot a^{-1}\right)_{a\in A}=\left((E\cdot a^{-1})\cap U\right)_{a\in A}
\]
qui est manifestement définissable.

\paragraph{Vérification d'hypothèse}
L'autre changement est  la vérification de l'inégalité en bas~\eqref{preuve majoration hypothèse en bas}. On pose 
$$
\nu_i= {\widehat{\mu_i}^{\sharp}}_{\vert K}.
$$
Comme 
\[
\nu_i(A'(b'_i))\geq \mu_i(\wV\cap\theta(U\cdot a_i)),
\]
on obtient
\begin{equation}\label{preuve majoration hypothèse en bas2}
\liminf \nu_i(A'(b'_i))\geq\liminf \mu_i(\wV\cap\theta(U\cdot a_i))>0.
\end{equation}

Les même calcul qu'en \ref{eq7.6}, en remplaçant~$\wV$ par~$\wV'=\wV\cdot a^{-1}$, et en modifiant~$f$ et~$h$ en conséquence donne
\begin{equation}\label{eq7.7}
\nu_i(A'(b'_i))=\int_{K} \,\,\,f\restriction_K~{{{\widehat{\mu}_i}^\sharp}\restriction_K}=
\int_{G}\,\,\, f~{\widehat{\mu}_i}^\sharp
\geq\int_{\Gamma\sous G} h \widehat{\mu}_i=\widehat{\mu}_i((\wV a_i^{-1})\cap \theta(U))=\mu_i(\wV\cap (Ua_i)).
\end{equation}

\paragraph{} Enfin il ne faut pas oublier de vérifier que l'usage de la Proposition~\ref{Proposition mesures}, et son Corollaire~\ref{Corollaire mesures}, invoqués
avec la nouvelle famille~\eqref{invoque définissabilité 2 DGU}, donne la conclusion cherchée, ce qui est le cas.

Plus précisément on déduira que l'ensemble
$ A'(b'_i)=\theta^{-1}(\wV)\cap (U\cdot a_i) a_i^{-1}$ contient un ouvert non vide de~$A(b_i)=\Gamma \cdot H_\infty\cdot h_i$. Il suffire de translater à droite par~$a_i$
pour obtenir la conclusion cherchée. Enfin le passage de la première conclusion à la seconde se fait sans changement.

\subsection{Preuve du Théorème \ref{teo1.6}}

Nous pouvons maintenant donner la preuve du Théorème \ref{teo1.6}. Nous commençons par le Lemme suivant qui nous permet de supposer que $M=\{1\}$.

\begin{lemme}
Soit $M\subset G$ un compact.
Le Théorème \ref{teo1.6} pour $S_{\Gamma,G,M}=\Gamma\backslash G/M$ est une conséquence du Théorème \ref{teo1.6} pour $S=\Gamma \backslash G$.
\end{lemme}
\begin{proof}
Soit $V$ une sous-variété analytique fermée de $S_{\Gamma,G,M}$ définissable dans $\R^{an,exp}$.  Soit 
$$
\alpha_M: S=\Gamma \backslash G\longrightarrow  S_{\Gamma,G,M} \ \  \mbox{ l'application } \Gamma g\mapsto \Gamma gM.
$$
Comme $\alpha_M$ est propre et définissable dans $\RR^{alg}$ pour les structures de variétés   $\RR^{alg}$-définissables  sur $S_{\Gamma,G,M}$ et $S$ introduites par Bakker-Klingler-Tsimerman (c.f section \ref{Vom}), $\widehat{V}=\alpha_M^{-1}(V)$ est analytique fermée et définissable dans $\R^{an,exp}$.

Soit 
$$
Z_i=\Gamma_{H_i}\backslash H_i(\RR)^+ g_i\cdot M
$$
une suite de sous-variétés homogènes de $V$. Alors 
$$
\widehat{Z}_i=\Gamma_{H_i}\backslash H_i(\RR)^+ g_i
$$
est une suite de sous-variétés homogènes de $\widehat{V}\subset S$. D'après le Théorème \ref{teo1.6} pour $\widehat{V}$ dans $S$, en passant éventuellement à une sous-suite, on peut supposer l'existence d'un sous-groupe algébrique $\HH_{\infty}$  tel que $\HH_i\subset \HH_{\infty}$ et tel que
$$
\widehat{Z}'_i:= \Gamma_{H_{\infty}}\backslash H_i(\RR)^+ g_i \subset \widehat{V}. 
$$
On en déduit l'existence de la suite de sous-variétés homogènes
$$
{Z}'_i:= \Gamma_{H_{\infty}}\backslash H_i(\RR)^+ g_i \cdot M\subset  V
$$
prévue  par le Théorème \ref{teo1.6} pour $V$ dans $S_{\Gamma,G,M}$.
\end{proof}

On suppose donc que $M=\{1\}$ et on se donne une suite de sous-variétés homogènes 
$$
(Z_i)_{i\in \N}=(\Gamma_{H_i}\backslash H_i(\RR)^+ g_i)_{i\in \N}
$$
 de $\widehat{V}\subset S=\Gamma\backslash G$. Soit $(\mu_i)_{i\in \N}=(\mu_{\HH_i}\cdot g_i)_{i\in \N}$ la suite de mesure de type Ratner associée. 
D'après la discussion de la section \ref{sDGU} on peut suppose, en passant au besoin à une sous-suite,  qu'il existe un parabolique $\PP$ et une écriture de type DGU pour $\PP$ de la suite $(\mu_i)_{i\in \N}$. On a alors 
$$
g_i=n_i m_i a_i k_i\in G=NMAK
$$
avec $h_i=n_i m_i\in H_P =NM$ et $H_i\subset H_{P}$ pour tout $i$. 

\begin{lemme}
Il existe une partie définissable bornée $U$ de $G$ telle que 
\begin{equation}
\liminf_{i\in\N}\mu_i(\wV\cap\theta(U\cdot a_i))>0,
\end{equation}
 et telle que $\theta^{-1}(\wV)\cap (U\cdot A^+_{\ge t})$ est définissable dans $\RR^{an, exp}$.
\end{lemme}
 \begin{proof}
 Notons que par hypothèse le support de $\mu_i$ est contenu dans $\widehat{V}$ de sorte que
 $$
 \mu_i(\wV\cap\theta(U\cdot a_i))= \mu_i(\theta(U\cdot a_i))
 $$
  D'après le Lemme \ref{bonK}, on peut choisir $U=V_P\subset H_P$ définissable borné vérifiant la première équation et tel que $\Sigma(\PP, V_P,t)=V_PA^+_{\ge t}K_{\infty}$ soit un ensemble de Siegel de $\G$ relativement à $\PP$. Les propriétés de définissabilité rappelées à la Proposition 
 \ref{defpi} assurent alors que 
 $$
 \theta^{-1}(\wV)\cap (U\cdot A^+_{\ge_t})\cdot K_{\infty}
 $$
  et par suite $\theta^{-1}(\wV)\cap (U\cdot A^+_{\ge_t})$ est définissable dans $\RR^{an,exp}$.
 \end{proof}

D'après le Théorème \ref{Charnière DGU},  $ \wV$ contient un ouvert non vide de  $\Gamma \backslash \Gamma H_{\infty}g_i$ pour  tout $i$. Par un argument de prolongement analytique 
$$
\Gamma \backslash \Gamma H_{\infty}g_i\subset \widehat{V}.
$$ 

Ceci termine la preuve du Théorème \ref{teo1.6}.

\section{Application aux variété arithmétiques: partie géométrique d'André-Oort}\label{sec9}

On reprend les notations des parties précédentes. On fixe en particulier un $\QQ$-groupe $\G$ semi-simple de type adjoint. On fixe un sous-groupe compact maximal $K_{\infty}$ de $G$.
 On suppose que l'espace symétrique $X:=G/K_{\infty}$ est hermitien. On fixe un point $x_0$ de $X$ dont le stabilisateur dans $G$ est $K_{\infty}$. On fixe un réseau arithmétique $\Gamma\subset \G(\Q)$ et on note
$$
\alpha=\alpha_{x_o}: G\longrightarrow X \ \ \ \ \ \mbox{  et } \beta=\beta_{x_o}: \Gamma \backslash G\longrightarrow  \Gamma\backslash X
$$
les applications données respectivement par $g\mapsto g.x_0=\alpha(g)$ et $\Gamma g\mapsto \Gamma g.x_0= \beta(\Gamma g)$

On a alors le diagramme commutatif suivant.

$$
\xymatrix{ 
&G \ar[rd]^{\theta} \ar[ld]_{\alpha_{x_0}} \\
X  \ar[rd]^{\pi} && \Gamma\backslash G \ar[ld]_{\beta_{x_0}} \\ 
& S=\Gamma\backslash X  
}
$$

On donne  dans cette section une preuve du  Théorème \ref{T1}. 

 \begin{theorem}
 Une sous-variété algébrique irréductible $V$ de $S$ qui contient un sous-ensemble Zariski-dense de sous-variétés faiblement spéciales de dimension positive est une sous variété faiblement spéciale ou se décompose en un produit faiblement spécial.
 \end{theorem}

La preuve passe par celle du Théorème \ref{T2}

Soit $V$ une sous-variété irréductible de $S$ qui admet un sous-ensemble Zariski dense de sous-variétés faiblement spéciales de dimension positive. On peut supposer que $V$ est Hodge générique.   
Soit en effet $S'$ la plus petite sous-variété spéciale de $S$ qui contient $V$. Il existe alors un $\Q$-sous-groupe semi-simple $\G'$ de $\G$  et un point $x_1$ de $X$ tels que
$S'= \Gamma'\backslash X'$ avec $X':=G'.x_1$ hermitien symétrique et $\Gamma':=\Gamma\cap \G'(\Q)\subset \G'(\Q)$ un réseau arithmétique. Dans cette situation $V$ est Hodge générique dans $S'$
et la conclusion du Théorème pour $V$ dans $S'$ est identique à celle de $V$ dans $S$. On suppose donc dans la suite que $V$ est Hodge générique dans $S$.

On peut alors construire une suite $(Z_i)_{i\in \N}$ de sous-variétés faiblement spéciales de $V$ de dimension positive telle que
\begin{itemize}

\item Pour tout $i\in\N$, $Z_i$ est maximal parmi les sous-variétés faiblement spéciales de $V$.
\item La suite $(Z_i)_{i\in \N}$ est {\bf générique} dans $V$: pour toute sous-variété stricte $W$ de $V$, l'ensemble
$$
\{i\in \N, Z_i\subset W\}
$$
est de cardinal fini. On utilise ici l'hypothèse que $V$ est irréductible.

\item La suite $(Z_i)_{i\in \N}$ est {\bf  stricte} dans $S$: Pour toute sous-variété spéciale $S'$ stricte de $S$,
$$
\{i\in \N, Z_i\subset S'\}
$$
est de cardinal fini. C'est une conséquence de la propriété précédente car comme $V$ est Hodge générique, pour toute variété spéciale $S'$ stricte de $S$, $V\cap S'$ est un fermé de Zariski stricte de $V$ donc 
$$
\{i\in \N, Z_i\subset S'\}=\{i\in \N, Z_i\subset V\cap S'\}
$$
est bien de cardinal fini.

\item Il existe un sous-groupe parabolique $\PP$ de $\G$ et une suite de mesure $(\mu_i)_{i\in \N}$ de type Ratner sur $\Gamma\backslash G$ de la forme DGU dans $G$, muni de $K_{\infty}=K_{x_0}$ et de $\PP$  avec les écritures 
$$
\mu_i=\mu_{\HH_i}\cdot g_i\qquad{g_i=n_i\cdot m_i\cdot a_i\cdot k_i} \qquad{h_i=n_im_i\in H_P}
$$
où

 \item[{[DGU~1]}]  La suite~\((\widehat{\mu_i})_{i\in \N}:=(\mu_{\HH_i}\cdot h_i)_{i\in \N}\) est de la forme Mozès-Shah:

\item[(M.-S.~1)] la suite~$(h_i)_{i\in \N}$ est convergente, de limite notée~$h_\infty$;
\item[(M.-S.~2)] parmi les sous-groupes algébriques de~$\G$  contenant une infinité des~$\HH_i$, il en existe un unique minimal, que nous noterons~$\HH_\infty=\limsup_{i\rightarrow \infty} \HH_i$.
 On a de plus $\HH_i\subset \HH_P$ pour tout $i\ge 0$.

\item[{[DGU~2]}] la suite~$(a_i)_{i\in \N}$ vérifie
\[
\forall \alpha\in\Phi(P,A_P), \lim_{i\to \infty}\alpha(a_i)=\infty
\]

\item On a $Z_i= \Gamma\backslash \Gamma H_ig_i.x_0$. La mesure canonique sur $Z_i$ est  $\nu_i=(\beta_{x_0})_{\star} \ \mu_i$ et on pose 
 $\widehat{\nu_i}=(\beta_{x_0})_{\star}\  \widehat{\mu_i}$

\item On pose $\wV=\beta_{x_0}^{-1}(V)\subset \Gamma \backslash G$.

\item Comme dans la preuve du Théorème \ref{teo1.6} donnée dans la section précédente,
Il existe une partie définissable bornée $U$ de $G$ telle que 
\begin{equation}
\liminf_{i\in\N}\mu_i(\wV\cap\theta(U\cdot a_i))>0,
\end{equation}
 et telle que $\theta^{-1}(\wV)\cap (U\cdot A^+_{\ge _t})$ est définissable dans $\RR^{an, exp}$.

\end{itemize}

D'après le Théorème \ref{Charnière DGU},   $ \wV$ contient un ouvert non vide de  $\Gamma \backslash \Gamma H_{\infty}g_i$ pour presque tout $i$. De plus $\HH_i\subset \HH_{\infty}$ pour presque tout $i$.
On en déduit que $V$ contient un ouvert non vide de $\Gamma \backslash \Gamma H_{\infty}g_i.x_0$ pour presque tout $i$. Par un argument de prolongement analytique on obtient les inclusions
\begin{equation}
Z_i\subset \Gamma \backslash \Gamma H_{\infty}g_i.x_0\subset V.
\end{equation}

Par le {\bf baby case} du Théorème d'Ax-Lindemann hyperbolique (Proposition \ref{BAL}) la clôture de Zariski 
$Z'_i$ de $\Gamma \backslash \Gamma H_{\infty}g_i.x_0$ est faiblement spéciale et on a en fait

\begin{equation}
Z_i\subset \Gamma \backslash \Gamma H_{\infty}g_i.x_0\subset  Z'_i\subset V.
\end{equation}

par notre hypothèse de maximalité, on obtient que
\begin{equation}
Z_i =\Gamma \backslash \Gamma H_{\infty}g_i.x_0=Z'_i
\end{equation}
d'où $\HH_i=\HH_{\infty}$ pour tout $i$ assez grand et
\begin{equation}
Z_i=\Gamma\backslash \Gamma H_{\infty}g_i.x_0=\Gamma\backslash \Gamma H_{\infty} x_i
\end{equation}
avec $x_i:=g_i.x_0$.

On est alors ramené à l'étude de l'ensemble $\cE(\HH_{\infty}) $ des sous-variétés faiblement spéciales de $S$ de la forme 
$\Gamma\backslash \Gamma H_{\infty} x$ pour un $x\in X$. Cette étude est faite dans la section 3 de \cite{Ullmo3} dont la  Proposition 3.5
est résumée dans la Proposition suivante.

\begin{proposition} \label{p3.4}
Soit $\HH$ un $\Q$-sous-groupe semi-simple de $\G$ tel que
$\cE(\HH)$ soit non vide. On suppose que $\HH \neq \G^{\rm der}$.

\item (i) Si $\HH$ n'est pas  un sous-groupe normal de $\G$
alors $\cE(\HH)$ est contenue dans une union finie de sous-vari\'et\'es sp\'eciales
strictes de $S$.
  
\item (ii) Si $\HH$ est un sous-groupe normal de $\G$,
 il existe une décomposition de  $X$ sous la
 forme \linebreak $X=X_{1}\times X_{2}$ en un produit de deux espaces symétriques hermitiens
 telle que $\cE(\HH)$ est l'ensemble des 
   $$
  \pi( {X_{1}\times \{x_{2}\}})
  $$
  avec $x_{2}\in X_{2}$.
 \end{proposition}

Si $\HH_{\infty}$ n'est pas normal dans $\G$, la première partie de la Proposition montre que les $Z_i$ sont contenus dans une union finie de sous-variétés spéciales strictes de $V$. Cela contredit notre hypothèse que la suite $(Z_i)_{i\in \NN}$ est stricte. On en déduit alors que $\HH_{\infty}$ est normal dans $\G$. et que 
$$
Z_i=   \pi( {X_{1}\times \{x_{i}\}})
$$
pour un $x_i\in X_2$. Comme les $Z_i$ sont Zariski denses dans $V$, on en déduit que $V$ est faiblement spéciale (si $\HH_{\infty}=\G$) ou se décompose en un produit faiblement spécial. Ceci termine la preuve du Théorème \ref{T1}

\appendix

\section{Equidistribution des sous-vari\'et\'es faiblement sp\'eciales horizontales dans les domaines de p\'eriodes }\label{apendix}

\centerline{Jiaming Chen, Rodolphe Richard et Emmanuel Ullmo}

\bigskip

Nous pr\'esentons dans cette annexe  une généralisation du Th\'eor\`eme  \ref{T1} dans le contexte des $\ZZ$-variations de structures de Hodge  polarisables ($\ZZ$-VHS).
\bigskip

\vspace{-1em}

\subsection{Sous-variétés faiblement spéciales des domaines de p\'eriodes} Dans cette sous-section, nous rappelons les d\'efinitions pertinentes de la th\'eorie de Hodge et des sous-vari\'et\'es sp\'eciales et faiblement sp\'eciales pour une $\ZZ$-variation  de structures de Hodge polarisables. Nous faisons r\'ef\'erence \`a \cite{Klingler17}, \cite{KO19} Section 4, ou à \cite{C21} Section 3 pour plus de d\'etails. 

Soit $\VV$ un $\ZZ$-VHS sur une vari\'et\'e quasi-projective complexe irr\'eductible lisse $S$. Soit $\Gbold$ son groupe de Mumford-Tate générique. On fixe un point Hodge g\'en\'erique $o\in S$. La structure de Hodge sur la fibre $V := \VV_{\QQ, o}$ induit un morphisme de $\RR$-groupes alg\'ebriques $x_o\colon \SS\to\Gbold_\RR$. Soit $\Dcal$ la $\Gbold(\RR)$-classe de conjugaison de $x_o$ et soit $\Dcal^+$ la composante connexe de $\Dcal$ contenant $x_o$. Le couple $(\Gbold, \Dcal^+)$ (voir \cite{C21} D\'efinition 3.1), est la \emph{donn\'ee de Hodge connexe g\'en\'erique} associ\'ee \`a $(S, \VV)$.

Soit $\Gam$ un r\'eseau arithm\'etique net de $\Gbold(\QQ)_+$, o\`u $\Gbold(\QQ)_+$ est le stabilisateur dans $\Gbold(\QQ)$ de $\Dcal^+$. Apr\`es passage \`a une recouvrement \'etale finie de $S$, on peut supposer que $\Gam$ contient le groupe de monodromie de $S$ pour $\VV$, \`a savoir l'image de $\pi_1(S,o)$ dans $GL(V_{\ZZ})$. Le quotient arithm\'etique 
\[
    \Gam\bs\Dcal^+ = \Gam\bs\Gbold(\RR)_+/M_o=\colon S_{\Gam,\Gbold, M_o},
\] 
est  la vari\'et\'e de Hodge connexe associ\'ee au triplet $(S, \VV, \Gam)$. Ici, $\Gbold(\RR)_+$ est le stabilisateur dans $\Gbold(\RR)$ de $\Dcal^+$, $M_o$ est l'intersection du sous-groupe d'isotropie de $x_o$ dans $\Gbold(\RR)$ avec $\Gbold(\RR)_+$ et son image dans $\Gbold^{\ad}(\RR)^+$ est compacte. La vari\'et\'e de Hodge connexe $\Gam\bs\Dcal^+$ est une vari\'et\'e analytique complexe mais n'a en g\'en\'eral pas de structure alg\'ebrique sous-jacente. La $\ZZ$-VHS $\VV$ peut alors \^etre d\'ecrite par son application de p\'eriode
\[
\psi\colon S\to \Gam\bs\Dcal^+,
\]
 qui est holomorphe, horizontal et nous avons le diagramme commutatif suivant
  \begin{displaymath}
	\begin{tikzpicture}
		\node (A) at (-2,2) {$\widetilde{S}$};
		\node (B) at (-2,0) {$S$};
		\node (C) at (0,2) {$\Dcal^+$};
		\node (D) at (0,0) {$\Gam\bs\Dcal^+$};
	
		\path[-latex]
		(A) edge node[left]{$p$} (B)
		(A) edge node[above]{$\widetilde{\psi}$} (C)
		(B) edge node[below]{$\psi$} (D)
		(C) edge node[right]{$\pi$} (D);
   \end{tikzpicture}
   \end{displaymath}
   o\`u $p$ est le recouvrement topologique universel de $S$, $\widetilde{\psi}$ le rel\`evement de $\psi$ et $\pi$  la projection naturelle.
   
   Dans \cite{Klingler17}, Klingler a introduit les notions de sous-vari\'et\'es sp\'eciales et faiblement sp\'eciales pour un $\ZZ$-VHS qui g\'en\'eralisent les notions correspondantes de la théorie des variétés de Shimura. Nous commen\c{c}ons par rappeler ces d\'efinitions.
   
   \begin{definition}
       [\cite{Klingler17}]
\begin{itemize}
	\item[(1)] une sous-vari\'et\'e \emph{sp\'eciale} de $\Gam\bs\Dcal^+$ est l'image d'une variété de Hodge connexe $Y$ par un morphisme de Hodge $Y\to\Gam\bs\Dcal^+$.
    \item[(2)] Une \emph{sous-vari\'et\'e faiblement sp\'eciale} de $\Gam\bs\Dcal^+$ est une sous-variété spéciale ou  l'image $\varphi(Y_1\times \{t_2\})$ pour un morphisme de Hodge $\varphi\colon Y_1\times Y_2\to\Gam\bs\Dcal^+$  avec $t_2\in Y_2$.	\end{itemize} 
   \end{definition}
   
   \begin{remarque}
   	   Un r\'esultat fondamental de Cattani, Deligne et Kaplan \cite{CattaniDeligneKaplan95} dit que si $W$ est une sous-vari\'et\'e sp\'eciale de $G\bs\Dcal^+$, alors chaque composante analytique irr\'eductible de $\psi^{-1}(W)$ est une sous-vari\'et\'e alg\'ebrique de $S$. R\'ecemment, Bakker, Klingler et Tsimerman \cite{BKT} ont donné une preuve alternative de cet énoncé en \'etablissant des propriétés de définassibilité  de l'application de p\'eriode dans une structure o-minimale et en appliquant un Lemme de Chow convenable dans ce contexte. Cett m\'ethode permet \'egalement de  montrer que si $W$ est une sous-vari\'et\'e faiblement sp\'eciale de $\Gam\bs\Dcal^+$, alors $\psi^{-1}(W)$ est une sous-vari\'et\'e alg\'ebrique de $S$ (voir \cite{KO19}, Proposition 4.8).
   \end{remarque}

Comme la d\'efinissabilit\'e de l'application de  p\'eriode sera utilis\'ee de mani\`ere cruciale dans la preuve de nos r\'esultats, nous rassemblons ici le th\'eor\`eme 1.1 de \cite{BKT} dans la g\'en\'eralit\'e que nous voulons pour la commodit\'e :

\begin{theorem}[Bakker, Klingler and Tsimerman \cite{BKT}, Thm. 1.1] \label{BKT20}

\begin{itemize}
	\item[(1)] Il existe une structure de vari\'et\'e  $\RR{\alg}$-définissable sur la vari\'et\'e de Hodge connexe $\Gam\bs\Dcal^+$. 
\item[(2)] Si on muni  $S$ (resp. $\Gam\bs\Dcal^+$)  de la  structure de vari\'et\'e $\RR^{\an, \exp}$-d\'efinissable \'etendant la structure de vari\'et\'e $\RR^{\alg}$-d\'efinissable sur $S$  (resp $\Gam\bs\Dcal^+$) provenant de sa structure de variété alg\'ebrique complexe (resp. provenant de la structure  de vari\'et\'e  $\RR^{\alg}$-définissable sur $\Gam\bs\Dcal^+$ donn\'ee par la partie (1)), l'application de  p\'eriode $\psi\colon S\to\Gam\bs\Dcal^+$ est $\RR{\an, \exp}$-d\'efinissable.
\item[(3)] Pour tout $q\in\Gbold(\QQ)_+$, la correspondance de Hecke :
\[
   T_q\colon \Gam\bs\Dcal^+\stackrel{\pi_1}{\longleftarrow}q^{-1}\Gam q\cap\Gam\bs\Dcal^+\stackrel{q\cdot}{\longrightarrow}\Gam\cap q\Gam q^{-1}\bs\Dcal^+\stackrel{\pi_2}{\longrightarrow}\Gam\bs\Dcal^+,
\]
est $\RR{\alg}$-d\'efinissable par rapport \`a la structure d\'efinissable sur $\Gam\bs\Dcal^+$ donn\'ee par la partie (1). Ici, $\pi_1$ et $\pi_2$ sont les projections \'etales finies naturelles et l'application $q\cdot$ est la multiplication \`a gauche par $q$.
\end{itemize}	
\end{theorem}

\begin{definition}
	Une sous-vari\'et\'e \emph{sp\'eciale} (resp. \emph{faiblement sp\'eciale}) de $S$ pour $\VV$ est  une composante irr\'eductible de $\psi^{-1}(W)$ pour une  sous-vari\'et\'e sp\'eciale (resp. faiblement sp\'eciale) $W$ de $\Gam\bs\Dcal^+$.
\end{definition}

\begin{definition}[\cite{KO19} Definition 1.3]
 Une sous-vari\'et\'e irr\'eductible ferm\'ee $Z\subseteq S$ pour $\VV$ est dite \emph{de dimension de p\'eriode positive pour $\VV$} si son groupe de monodromie alg\'ebrique pour $\VV$ est non trivial ; ou de mani\`ere \'equivalente si l'application de p\'eriode $\psi$ ne contracte pas $Z$ en un point de $\Gam\bs\Dcal^+$. 	
\end{definition}

\begin{definition}
    Une sous-vari\'et\'e faiblement sp\'eciale $W$ de $\Gam\bs\Dcal^+$ est dite \emph{horizontale} si $W$ est contenue dans $\overline{\psi(S)}$ (la fermeture topologique de $\psi(S)$ dans $\Gam\bs\Dcal^+)$.
    
	Une sous-vari\'et\'e faiblement sp\'eciale $Z$ de $S$ est dite \emph{horizontale} pour $\VV$ si $Z$ est une composante irr\'eductible de $\psi^{-1}(W)$ pour une  sous-vari\'et\'e faiblement sp\'eciale horizontale $W$ de $\Gam\bs\Dcal^+$. 	
	
	Nous d\'esignons par $\WS^h(S,\VV)_+$ l'union des sous-vari\'et\'es faiblement sp\'eciales  de $S$, horizontales pour $\VV$,  strictes de dimension de p\'eriode positive.	
\end{definition}

\subsection{Enoncé du r\'esultat principal.}

Le résultat principal que nous avons en vue est
\begin{theorem}\label{equidis:VHS}
	Soit $\VV$ une $\ZZ$-VHS sur une vari\'et\'e quasi-projective complexe irr\'eductible lisse $S$. Si $\WS^h(S,\VV)_+\subseteq S$ est Zariski-dense dans $S$, alors (en passant au besoin à un revêtement fini étale de $S$) l'application de  p\'eriode
	\[
   \psi\colon S\to\Gam\bs\Dcal^+
\] 
se factorise en
\begin{displaymath}
	\begin{tikzpicture}
		\node (A) at (-4.5,0) {$S$};
		\node (B) at (0,0) {$\Gam_1\bs\Dcal_1^+\times\Gam_2\bs\Dcal_2^+$};
		\node (C) at (4.5,0) {$\Gam\bs\Dcal^+$,};
	
		\path[-latex]
		(A) edge node[above]{$(\psi_1, \psi_2)$} (B)
		(B) edge node[above]{$\varphi$} (C);
   \end{tikzpicture}
   \end{displaymath}
   o\`u $\Gam_1\bs\Dcal_1^+$ est une vari\'et\'e de Hodge connexe de type Shimura, $\varphi$ est un morphisme de Hodge fini, $\psi_1$ est un morphisme alg\'ebrique dominant et  
$\psi(S)$ est dense dans 
\[
    \varphi(\Gam_1\bs\Dcal_1^+\times \psi_2(S)).
\]
\end{theorem}

Les ingr\'edients principaux de la preuve du Th\'eor\`eme \ref{equidis:VHS} sont le Th\'eor\`eme \ref{teo1.6} et le  "baby case" du th\'eor\`eme d'Ax-Lindemann  pour les  $\ZZ$-VHS suivant.

\begin{theorem}[Baby-Ax-Lindemann pour les $\ZZ$-VHS]\label{AL:baby} 
Soit $\Hbold$ un sous-groupe de $\Gbold$ de type $\Hcal$.
	Soit $W = \Gam\bs\Gam\Hbold(\RR)^+gM/M\subset \overline{\psi(S)}$ une sous-vari\'et\'e homog\`ene horizontale (ferm\'ee) de $\Gam\bs\Dcal^+$. Soit $Z$ une composante irr\'eductible de la clôture de Zariski de $\psi^{-1}(W)$. Alors $Z$ est une sous-vari\'et\'e faiblement sp\'eciale horizontale de $S$ pour $\VV$.	
\end{theorem}

On peut noter que cet énoncé n'est pas une conséquence directe du théorème d'Ax-Lindemann dans ce contexte à cause des hypothèses et conclusions d'horizontalité.  

\subsection{Description des sous-vari\'et\'es faiblement sp\'eciales de $S$ pour $\VV$} \label{descrip:WS} Les sous-vari\'et\'es faiblement sp\'eciales de $S$ pour $\VV$ peuvent \^etre d\'ecrites en utilisant les groupes de monodromie alg\'ebrique.

\begin{definition}
    Soit $Z$ une sous-vari\'et\'e alg\'ebrique irr\'eductible ferm\'ee de $S$ et soit $\nu\colon\hat{Z}\to Z$ sa normalisation. Le \emph{groupe de monodromie alg\'ebrique} $\Hbold_Z$ de $Z$ pour $\VV$ est d\'efini comme \'etant la fermeture de Zariski dans $\GL(V)$ de la monodromie du syst\`eme local $\nu^*\VV$ sur $\hat{Z}$. 	
\end{definition}

Soit $Z$ une sous-vari\'et\'e alg\'ebrique irr\'eductible ferm\'ee de $S$. Nous avons alors un diagramme commutatif
\begin{displaymath}
	\begin{tikzpicture}
		\node (A) at (-2,2) {$\hat{Z}$};
		\node (B) at (-2,0) {$Z$};
		\node (C) at (0,2) {$\Gamma_Z\bs\Dcal_Z^+$};
		\node (D) at (0,0) {$\Gam\bs\Dcal^+$};

		\path[-latex]
		(A) edge node[left]{$\nu$} (B)
		(A) edge node[above]{$\psi_Z$} (C)
		(B) edge node[below]{$\psi | _Z$} (D)
		(C) edge node[right]{$\iota_Z$} (D);
    \end{tikzpicture}
\end{displaymath}
o\`u $(\Gbold_Z,\Dcal_Z^+)$ est la donn\'ee de Hodge connexe g\'en\'erique de la restriction de $\VV$ au lieu lisse de $Z$ et $\Gam_Z = \Gam\cap\Gbold_Z(\QQ)_+$. Par le th\'eor\`eme de monodromie d'Andr\'e-Deligne \cite{An92} , $\Hbold_Z$ est un sous-groupe normal de $\Gbold_Z^\der$. Puisque $\Gbold_Z$ est r\'eductif, il existe un sous-groupe normal $\Gbold_Z^\prime$ de $\Gbold_Z$ et  un produit presque direct
\[
   \Gbold_Z = \Hbold_Z\Gbold_Z^\prime,
\]
qui induit une  d\'ecomposition de la donn\'ee de Hodge connexe adjointe $(\Gbold_Z^{\ad},  \Dcal_Z^+)$
\begin{equation*}
  (\Gbold_Z^{\ad},  \Dcal_Z^+) = (\Hbold_Z^{\ad}, \Dcal_{\Hbold_Z}^+)\times(\Gbold_Z^{\prime\ad}, \Dcal_{\Gbold_Z^\prime}^+).
\end{equation*}
On en déduit un morphisme de Hodge fini
\begin{equation*}
   \alpha_Z\colon\Gam_{\Hbold_Z}\bs\Dcal_{\Hbold_Z}^+\times\Gam_{\Gbold_Z^\prime}\bs\Dcal_{\Gbold_Z^\prime}^+\to \Gam_{Z}\bs\Dcal_{Z}^+,
\end{equation*}
et l'application de  p\'eriode $\psi_Z$ se factorise en
 \begin{equation*}
	\begin{tikzpicture}
		\node (A) at (-4.5,0) {$\hat{Z}$};
		\node (B) at (0,0) {$\Gam_{\Hbold_Z}\bs\Dcal_{\Hbold_Z}^+\times\Gam_{\Gbold_Z^\prime}\bs\Dcal_{\Gbold_Z^\prime}^+$};
		\node (C) at (4.5,0) {$\Gam_Z\bs\Dcal_Z^+$.};
	
		\path[-latex]
		(A) edge node[above]{$(\psi_Z^{nt}, \psi_Z^t)$} (B)
		(B) edge node[above]{$\alpha_Z$} (C);
   \end{tikzpicture}
   \end{equation*}
Le Lemme 4.2  de \cite{KO19} assure que de la projection $\psi_Z^t\colon\hat{Z}\to\Gam_{\Gbold_Z^\prime}\bs\Dcal_{\Gbold_Z^\prime}^+$ est constante. 
Donc si $Z$ est une sous-vari\'et\'e faiblement sp\'eciale de $(S,\VV)$, alors $Z$ est une composante irr\'eductible de 
\[
    \psi^{-1}(\iota_Z\circ\alpha_Z(\Gam_{\Hbold_Z}\bs\Dcal_{\Hbold_Z}^+\times\psi_t(\hat{Z}))) = \psi^{-1}(\pi(\Dcal^+_{\Hbold_Z}\times\{x_Z^{\dprime}\})),
\]
o\`u $x_Z^\dprime\in\Dcal_{\Gbold_Z^\prime}^+$ est un rel\`evement arbitraire à $\Dcal_{\Gbold_Z^\prime}^+$ du point $y^{\dprime}_Z:=\psi_Z^t(\hat{Z})$ .

On en déduit la description suivante des sous-vari\'et\'es faiblement sp\'eciales de $S$ pour $\VV$ : 

Soit $Z\subseteq S$ une sous-vari\'et\'e alg\'ebrique irr\'eductible ferm\'ee de $S$. Alors $Z$ est faiblement sp\'eciale pour $\VV$ si et seulement si $Z$ est maximale parmi toutes les sous-vari\'et\'es alg\'ebriques irr\'eductibles ferm\'ees de $S$ ayant le même groupe de monodromie alg\'ebrique que $Z$. 

Soient $x_Z^\prime \in\Dcal_{\Hbold_Z}^+$  et $x_Z := (x_Z^\prime, x_Z^\dprime)$. Alors
\[
    \Dcal^+_{\Hbold_Z}\times\Dcal_{\Gbold_Z^\prime}^+ = \Gbold_Z(\RR)^+\cdot x_Z =  \Hbold_Z(\RR)^+\cdot x_Z^\prime\times\Gbold_Z\prime(\RR)^+\cdot x_Z^\dprime,
\]
\[  
   \Dcal^+_{\Hbold_Z}\times\{x_Z^\dprime\} = \Hbold_Z(\RR)^+\cdot x_Z = \Hbold_Z(\RR)^+\cdot x_Z^\prime\times\{x_Z^\dprime\}.
\]
Donc
\[ \pi(\Dcal^+_{\Hbold_Z}\times\{x_Z^{\dprime}\}) = \pi(\Hbold_Z(\RR)^+\cdot x_Z) = \Gam\bs\Gam\Hbold(\RR)^+g_Z M/M
\]
pour un  $g_Z$ dans $\Gbold(\RR)^+$.

La description ci-dessus est canonique modulo $\Gam$-conjuguaison : si on fixe un domaine fondamental $\Fcal$ pour les actions de $\Gam$ sur $\Dcal^+$ et que l'on exige que $x_Z\in\Fcal$, alors on obtient une description canonique des sous-vari\'et\'es faiblement sp\'eciales.

\begin{remarque}
	Soit $Z$ une sous-vari\'et\'e horizontale faiblement sp\'eciale  de $(S,\VV)$. Alors $\Dcal_{\Hbold_Z}^+$ est un domaine sym\'etrique hermitien 
et $\pi(\Hbold_Z(\RR)^+\cdot x_Z)$ est une sous-vari\'et\'e homog\`ene (voir sous-section \ref{QAG}) de $\Gam\bs\Dcal^+$ contenue dans l'image de p\'eriode.
\end{remarque}

\begin{remarque}
    En sp\'ecialisant la discussion ci-dessus \`a $Z=S$, on peut ignorer la partie de monodromie  triviale du groupe g\'en\'erique de Mumford-Tate $\Gbold$ . On supposera donc dans la suite  que le groupe de monodromie de $S$ pour $\VV$ est Zariski dense dans $\Gbold$.	
\end{remarque}

\subsection{Les preuves} 

Pour les preuves des Théorèmes  \ref{equidis:VHS} et \ref{AL:baby}, nous pouvons et nous allons supposer que l'application de  p\'eriode $\psi$ pour $\VV$ est propre.  En effet, soit $\bar{S}$ une compactification projective lisse de $S$ avec $\bar{S}\bs S$ un diviseur à croisements normaux. Soit $S^\prime$ l'union de $S$ et des points  de $\bar{S}\bs S$ autour desquels les monodromies locales sont d'ordre fini. Par un résultat de Griffiths (\cite{GriPerIII} Th\'eor\`emes 9.5 et 9.6), l'application de  p\'eriode $\psi$ s'\'etend holomorphiquement en une application propre $\psi^\prime\colon S^\prime\to \Gam\bs\Dcal^+$ et l'image $\psi^\prime(S^\prime)$ contient $\psi(S)$ comme complémentaire d'une sous-vari\'et\'e analytique. Les conclusions des Th\'eor\`emes \ref{equidis:VHS} et \ref{AL:baby} pour $\psi$ se déduisent alors  de celles pour $\psi^\prime$.

Soit $Z$ une  sous-vari\'et\'e faiblement sp\'eciale de $S$ pour $\VV$. On note   $\mu_Z$ sur $\Gam\bs\Dcal^+$
sa mesure de probabilit\'e de Borel canonique associée. 
 Le support de $\mu_Z$ est
  $$
  \supp(\mu_Z) = \pi(\Hbold_Z(\RR)^+\cdot x_Z)
  $$ 
  et $\mu_Z$ est l'image sous la projection (induite par $x_Z$)
\[
   \Gam\bs G\to\Gam\bs\Dcal^+
\]
de la mesure homog\`ene $\mu_{\Hbold_Z}\cdot g_Z$ sur $\Gam\bs G$ qui est de type Ratner (voir D\'efinition \ref{écriture}).
\\

Admettons provisoirement le Th\'eor\`eme \ref{AL:baby} et prouvons le Th\'eor\`eme \ref{equidis:VHS}.

\begin{proof}[Preuve du Th\'eor\`eme \ref{equidis:VHS}] 

Par l'hypothèse on dispose d'une suite $(Z_n)_{n\in\NN}$ de sous-vari\'et\'es  faiblement sp\'eciales, horizontales,  à dimension de p\'eriode positive pour $(S,\VV)$ telle que
\begin{itemize}
	\item[(1)] pour tout $n\in\NN$, $Z_n$ est maximal parmi les sous-vari\'et\'es faiblement sp\'eciales horizontales à dimensions de p\'eriode positives pour $(S,\VV)$ ;
	\item[(2)] la suite $(Z_n)_{n\in \NN}$ est g\'en\'erique ;
	\item[(3)] soit $(W_n)_{n\in\NN}$ la suite correspondante de sous-vari\'et\'es faiblement sp\'eciales horizontales de $\Gam\bs\Dcal^+$. On peut \'ecrire
	\[ 
	    W_n = \Gam\bs\Gam\Hbold_n(\RR)^+ g_n M/M\subseteq\psi(S),
    \]
	o\`u $\Hbold_n$ est le groupe de monodromie alg\'ebrique de $Z_n$ (qui est de type $\Hcal$ ). 
	
	\end{itemize}
	
Par le Th\'eor\`eme \ref{BKT20} (2),  $\psi(S)$ est une sous-vari\'et\'e analytique $\RR{\an,\exp}$- d\'efinissable de $\Gam\bs\Dcal^+$.		Le Th\'eor\`eme \ref{teo1.6} assure alors
	 qu'il existe  un sous-groupe $\Hbold_\infty$ de $\Gbold$,  de type $\Hcal$, tel que pour tout $n\in\NN$, $\Hbold_n\subseteq\Hbold_\infty$ et tel que $\psi(S)$ contient les sous-vari\'et\'es homog\`enes
	\[
		       W_n^\prime = \Gam\bs\Gam\Hbold_\infty(\RR)^+g_n M/M\subseteq\psi(S).
    \]

Soit $Z_n^\prime$ une composante irr\'eductible de $\overline{\psi^{-1}(W_n^\prime)}^\Zar$ contenant $Z_n$. Par le baby case du th\'eor\`eme d'Ax-Lindemann pour les $\ZZ$-VHS (Th\'eor\`eme \ref{AL:baby}), $Z_n^\prime$ est une sous-vari\'et\'e faiblement sp\'eciale horizontale de $(S,\VV)$ et $\overline{\psi^{-1}(W_n^\prime)}^\Zar$ est une union finie de sous-vari\'et\'es faiblement sp\'eciales horizontales. Donc $W_n^\prime$ est contenu dans une union finie de sous-vari\'et\'es horizontales faiblement sp\'eciales de $\Gam\bs\Dcal^+$. Ainsi	
\[
    \Hbold_\infty\subseteq\Hbold_{Z_n^\prime}.
\]
Mais par la maximalit\'e de $Z_n$, on a pour tout $n\in\NN$
\[
	Z_n = Z_n^\prime
\]
et donc
\[
    \Hbold_n = \Hbold_\infty.
\]
Par cons\'equent
\[
		   W_n = \Gam\bs\Gam\Hbold_n(\RR)^+ g_n M/M = \Gam\bs\Gam\Hbold_\infty(\RR)^+ g_n M/M.
\]

\begin{lemme}
$\Hbold_\infty$ est un sous-groupe normal de $\Gbold$.
\end{lemme}

\begin{proof}
Soit $\Ecal(\Hbold_\infty)$ l'ensemble des sous-vari\'et\'es faiblement sp\'eciales de $(S,\VV)$ qui sont des composantes irr\'eductibles de $\psi^{-1}(\pi(\Hbold_\infty(\RR)^+\cdot x)$ pour un  point $x\in\Fcal$. Comme  $\Hbold_\infty$ est de type Ratner et  $\pi(\Hbold_\infty(\RR)^+\cdot x)$ est faiblement spécial, $\Hbold_\infty$ est un $\QQ$-sous-groupe algébrique semi-simple de $\Gbold$. Supposons que $\Hbold_\infty$ n'est pas normal dans $\Gbold$. Soit $\widehat{\Hbold}_\infty:=\Hbold_\infty\cdot\Zcent_\Gbold(\Hbold_\infty)^\circ$. 

Soit $x\in\Fcal$ tel que  $\psi^{-1}(\pi(\Hbold_\infty(\RR)^+\cdot x)$ soit une union finie de composantes irréductibles dans $\Ecal(\Hbold_\infty)$.  Soit $Z_x$ une composante irr\'eductible de $\psi^{-1}(\pi(\Hbold_\infty(\RR)^+\cdot x)$ et soit $\Gbold_x$ le groupe de Mumford-Tate générique de $\VV |_{Z_x^{\circ}}$, la restriction de $\VV$ au lieu lisse de $Z_x$.  Alors  $(\Gbold_x, \Dcal_x := \Gbold_x\cdot x)$ est une sous-donnée de Hodge de $(\Gbold,\Dcal)$ et $\Hbold_\infty$ est un sous-groupe normal de $\Gbold_x$.

Donc $\widehat{\Hbold}_\infty$ est un sous-groupe de $\Gbold$ contenant $\Gbold_x$.
Comme $\Zcent_{\Gbold}(\Gbold_x)^\circ(\RR)/\Zcent^\circ(\Gbold)(\RR)$ est compact, on en déduit que $\Zcent_\Gbold(\widehat{\Hbold}_\infty)(\RR)^\circ/\Zcent(\Gbold)^\circ(\RR)$ est compact, donc par le Lemme 5.1 de \cite{EMS97}, $\widehat{\Hbold}_\infty$ est r\'eductif. Notez qu'il n'existe qu'un nombre fini de classes de conjugaison de \linebreak $h\colon\SS_\CC\to\widehat{\Hbold}_{\infty,\CC}$. Soit $x\colon\SS\to\widehat{\Hbold}_{\infty,\RR}$ (sur $\RR$). Alors le nombre de formes r\'eelles de $x$ est la cardinalit\'e de 
	\[
    \ker(H^1(\RR, L_x)\to H^1(\RR,\widehat{\Hbold}_{\infty}(\CC))),
    \]
    qui est finie. Ici, $L_x$ est le stabilisateur de $x$ dans $\widehat{\Hbold}_{\infty}(\CC)$. Il n'existe donc qu'un nombre fini choix de $\Dcal_{\widehat{\Hbold}_{\infty}}$ tels que $(\widehat{\Hbold}_{\infty},\Dcal_{\widehat{\Hbold}_{\infty}})$ soit un sous-donn\'ee de Hodge de $(\Gbold,\Dcal)$.
    Il existe donc  une sous-vari\'et\'e ferm\'ee stricte (une union finie de sous-vari\'et\'es sp\'eciales strictes de $(S,\VV)$) contenant tous les membres de $\Ecal(\Hbold_\infty)$, ce qui contredit la g\'en\'ericit\'e de la suite $(Z_n)_{n\in\NN}$. 	
\end{proof}

On peut maintenant terminer la preuve du th\'eor\`eme \ref{equidis:VHS}. 

Comme $\Hbold_\infty$ est un sous-groupe normal de $\Gbold$, on a a une d\'ecomposition de la donn\'ee de Hodge connexe adjointe $(\Gbold^{\ad}, \Dcal^+)$.
\[
  (\Gbold^{\ad},  \Dcal^+) = (\Hbold_\infty^{\ad}, \Dcal_{\Hbold_\infty}^+)\times(\Gbold^{\prime\ad}, \Dcal^{\prime+}).
\]
et en passant au besoin à un revêtement étale fini de $S$, une factorisation de l'application de p\'eriode $\psi$
\begin{displaymath}
\begin{tikzpicture}
		\node (A) at (-4.5,0) {$S$};
		\node (B) at (0,0) {$\Gam_{\Hbold_\infty}\bs\Dcal_{\Hbold_\infty}^+\times\Gam^\prime\bs\Dcal^{\prime+}$};
		\node (C) at (4.5,0) {$\Gam\bs\Dcal^+$.};
	
		\path[-latex]
		(A) edge node[above]{$(\psi_1, \psi_2)$} (B)
		(B) edge node[above]{$\varphi$} (C);
   \end{tikzpicture}
   \end{displaymath}
   On a alors
   \[
      W_n = \varphi(\Gam_{\Hbold_\infty}\bs\Dcal_{\Hbold_\infty}^+\times\{x_n^\prime\}) = \pi(\Hbold_\infty(\RR)^+\cdot z_n)\subseteq\psi(S)
   \]
pour un certain $z_n = (x_n, x_n^\prime)\in\Fcal$. La vari\'et\'e de Hodge connexe $\Gam_{\Hbold_\infty}\bs\Dcal_{\Hbold_\infty}^+$ est donc de type Shimura et $\psi_1$ est alg\'ebrique dominante.	

Soit $V^\prime$ le sous-ensemble de $\Gam^\prime\bs\Dcal^{\prime+}$ formé des points $x^\prime$ tels que $\varphi(\Gam_{\Hbold_\infty}\bs\Dcal_{\Hbold_\infty}^+\times\{x^\prime\})$ est contenu dans $\psi(S)$. Comme la suite $(Z_n)_{n\in\NN}$ est générique,  $\psi_2^{-1}(V^\prime)$ est Zariski dense dans $S$. On en déduit que
\[
	\psi(S) = \varphi(\Gam_{\Hbold_\infty}\bs\Dcal_{\Hbold_\infty}^+\times\psi_2(S)).\qedhere
\]	
\end{proof}

\begin{proof}[Preuve du Th\'eor\`eme \ref{AL:baby}]
	Nous conservons les notations de la sous-section \ref{descrip:WS}. Nous devons montrer que $Z$ est une composante irr\'eductible de $\psi^{-1}(\pi(\Dcal^+_{\Hbold_Z}\times\{x_Z^{\dprime}\})$.
	En rempla\c{c}ant $S$ par une composante irr\'eductible de $\psi^{-1}(\pi(\Dcal^+_{\Hbold_Z}\times\{x_Z^{\dprime}\})$ contenant $Z$, nous devons alors montrer que $Z =S$ et que $S$ est horizontal. L'application de p\'eriode $\psi$ admet la factorisation
	 \begin{displaymath}
	\begin{tikzpicture}
		\node (A) at (-4.5,0) {$S$};
		\node (B) at (0,0) {$\Gam_{\Hbold_Z}\bs\Dcal_{\Hbold_Z}^+\times\Gam_{\Gbold_Z^\prime}\bs\Dcal_{\Gbold_Z^\prime}^+$};
		\node (C) at (4.5,0) {$\Gam\bs\Dcal^+$.};
	
		\path[-latex]
		(A) edge node[above]{$(\psi_Z^{nt}, \psi_Z^t)$} (B)
		(B) edge node[above]{$\alpha_Z$} (C);
   \end{tikzpicture}
   \end{displaymath} 
    La projection $\psi_Z^t\colon S\to\Gam_{\Gbold_Z^\prime}\bs\Dcal_{\Gbold_Z^\prime}^+$ est constante. On en déduit que 
      $$
      \psi^{-1}(\pi(\Dcal^+_{\Hbold_Z}\times\{x_Z^{\dprime}\}) = (\psi^{nt}_Z)^{-1}(\Gam_{\Hbold_Z}\bs\Dcal_{\Hbold_Z}^+)
      $$
       On peut ainsi remplacer $\psi$ par $\psi^{nt}_Z$ et supposer que le groupe de monodromie de $Z$ est dense dans $\Gbold$.
    
    Soit $U$ une composante analytique irr\'eductible de $\psi(\overline{\psi^{-1}(W)}^\Zar)$ contenant $W$.
    
    Soient $q\in\Hbold(\QQ)^+$ et $T_q$ la correspondance de Hecke induite sur $\Gam\bs\Dcal^+$. Par le Th\'eor\`eme \ref{BKT20} (3), la correspondance de Hecke $T_q$ est $\RR^\alg$-d\'efinissable. Donc $T_q(U)$ est $\RR^{\an,\exp}$ d\'efinissable et analytique complexe, et 
\[
   \psi^{-1}(T_q(U))
\]
est $\RR{\an,\exp}$-d\'efinissable et analytique complexe. Le Th\'eor\`eme de Chow d\'efinissable nous assure alors que
\[
   \psi^{-1}(T_q(U))
\]
est une sous-vari\'et\'e alg\'ebrique de $S$. 
 Comme $W\subseteq T_q(W)$, $\psi^{-1}(W) \subseteq \psi^{-1}(T_q(U))$ et  
 $$
 \overline{\psi^{-1}(W)}^{Zar}\subseteq\psi^{-1}(T_q(U)).
 $$ Par cons\'equent, $U\subseteq T_q(U)\cap\psi(S)$.
Le Th\'eor\`eme \ref{AL:baby} d\'ecoulera alors du Lemme suivant. 	 
\end{proof}

\begin{lemme}\label{Const:Nori}  Il existe une constante $C(U)$, appel\'ee \emph{constante de Nori} de $U$ ayant la propriété suivante. Soit $p>C(U)$ un nombre premier $p>C(U)$.  Soit $q\in\Hbold(\QQ)^+$ tel que pour tout $\ell\neq p$, $g_\ell\in\Gam\cap\Gbold(\ZZ_\ell)$ et tel que les projections de $g_p$ sur chaque facteur $\QQ$-simple de $\Gbold^{\ad}$ ne sont pas contenues dans des sous-groupes compacts. Alors  $T_q(U)$ est irr\'eductible et les orbites de Hecke de $T_q+T_{q^{-1}}$ sont  denses dans $\Gam\bs\Dcal^+$ pour la topologie Archimédienne.  	
\end{lemme}

\begin{remarque}
	Dans le cas de  vari\'et\'e de Shimura, l'existence de la constante de Nori, l'irr\'eductibilit\'e et la densit\'e des orbites de Hecke ont \'et\'e prouv\'ees par Edixhoven et Yafaev (\cite{EY03} Th\'eor\`emes 5.1 et 6.1). L'existence de $q$ a \'et\'e détaillée  dans (\cite{UY14} Lemme 6.1). Les arguments sont essentiellement de nature topologique et peuvent \^etre facilement adapt\'es \`a notre cas. Seule l'existence de la constante de Nori n\'ecessite de petites modifications, que nous indiquons ici et renvoyons \`a \cite{EY03} et \cite{UY14} pour les d\'etails.
\end{remarque}

\begin{proof}[Esquisse de la preuve du Lemme \ref{Const:Nori}]
Consid\'erons la correspondance de Hecke
\[
   T_q\colon \Gam\bs\Dcal^+\stackrel{\pi_{q,1}}{\longleftarrow}q^{-1}\Gam q\cap\Gam\bs\Dcal^+\stackrel{\pi_{q,2}}{\longrightarrow}\Gam\bs\Dcal^+,
\]
o\`u $\pi_{q,1}$ est la projection \'etale naturelle finie. Soit $U_q := \pi_{q,1}^{-1}(U)$. Alors pour montrer que $T_q(U)$ est irr\'eductible, il suffit de montrer que $U_q$ est irr\'eductible. Puisque $U_q$ est analytique, il suffit de montrer que le lieu lisse $U_q^\circ$ est connexe. Si on se limite  aux lieux lisses, l'application $U_q^\circ\to U^\circ$ est une application de recouvrement. Ce recouvrement correspond au $\pi_1(U^\circ)$-ensemble 
$$
q^{-1}\Gam q\cap\Gam\bs\Gamma
$$
 et $U_q^\circ$ est connexe si et seulement si $\pi_1(U^\circ)$ agit sur $q^{-1}\Gam q\cap\Gam\bs\Gamma$ transitivement.
	
	Notez que nous avons des homomorphismes
\[
   \pi_1(Z) \to \pi_1(U) \to \Gam\subset\Gbold(\ZZ)\subset\GL(V_\ZZ).
\]
Puisque $\Gam$ est un sous-groupe arithm\'etique de $\Gbold(\QQ)$, il est Zariski dense dans $\Gbold$. Donc $\pi_1(U^\circ)$ et $\Gam$ ont la m\^eme fermeture de Zariski dans le $\ZZ$-sch\'ema en groupes $\GL(V_\ZZ)$ (ici nous avons utilis\'e la r\'eduction dans la preuve du Th\'eor\`eme \ref{AL:baby} que le groupe de monodromie de $Z$ pour $\VV$ est Zariski-dense dans $\Gbold$). Le reste de la preuve est la m\^eme que dans \cite{EY03} Th\'eor\`emes 5.1 et 6.1. La preuve de l'existence de $q$ peut \^etre reproduite mot pour mot de \cite{UY14} Lemma 6.1.
\end{proof}

\end{document}